\documentclass[12pt]{amsart}
\usepackage{bm,color,psfrag,epsfig,url,hhline,amsthm
}
\usepackage[dvipsnames]{xcolor}

\numberwithin{figure}{section}
\numberwithin{table}{section}
\newtheorem{theorem}{Theorem}[section]

\newtheorem*{conj}{Conjecture}
\newtheorem{lemma}[theorem]{Lemma}
\newtheorem{prop}[theorem]{Proposition}

\theoremstyle{definition}
\newtheorem{definition}[theorem]{Definition}
\newtheorem{example}[theorem]{Example}

\newtheorem{cor}[theorem]{Corollary}
\newtheorem{question}[theorem]{Question}

\theoremstyle{remark}
\newtheorem{remark}[theorem]{Remark}
\numberwithin{equation}{section}
\newfont{\tap}{tap scaled 650}

\def \V{{\mathbf V}}

\def \H{{\mathbb H}}

\def \R{{\mathbb R}}
\def \Re{{\mathrm {Re}}}
\def \Im{{\mathrm {Im}}}

\def \Z{{\mathbb Z}}

\def \O{{\mathcal O}}

\def \C{{\mathrm C}}
\def \CC{{\mathbb C}}
\def \A{{\mathcal A}}
\def \[{[ }
\def \]{] }

\def \T{{\mathcal T}}
\def \P{\mathcal{P}}
\def \p{\mathrm{p}}

\def\G{\Ca}
\def\t{\widetilde}

\def\cc{\mathbf{c}}
\def\dd{\mathbf{d}}

\def \s{\theta }

\def \V{{\mathcal V}}

\def \Ca{{\mathcal C}}

\def \k{{k}}

\begin{document}

\title[Acyclic cluster algebras and reflection groups]{Acyclic cluster algebras, reflection groups, and curves on a punctured disc}
\author{Anna Felikson and Pavel Tumarkin}
\address{Department of Mathematical Sciences, Durham University, Science Laboratories, South Road, Durham, DH1 3LE, UK}
\email{anna.felikson@durham.ac.uk, pavel.tumarkin@durham.ac.uk}
\thanks{AF was partially supported by EPSRC grant EP/N005457/1}

\begin{abstract}
We establish a bijective correspondence between certain non-self-intersecting curves in an $n$-punctured disc and  positive $\cc$-vectors of acyclic cluster algebras whose quivers have multiple arrows between every pair of vertices. As a corollary, we obtain a proof of Lee -- Lee conjecture~\cite{LL} on the combinatorial description of real Schur roots for acyclic quivers with multiple arrows, and give a combinatorial characterization of seeds in terms of curves in an $n$-punctured disc.
\end{abstract}

\maketitle
\setcounter{tocdepth}{1}
\tableofcontents

\section{Introduction and main results}

Given an acyclic quiver $Q$, {\em real Schur roots} are dimension vectors of indecomposable rigid representations of $Q$ over the path algebra $\k Q$, where $\k$ is an algebraically closed field. According to Kac~\cite{Kac}, real Schur roots are indeed positive roots of the root system $\Delta$ of the Kac-Moody algebra constructed by the generalized Cartan matrix defining the Tits quadratic form of $Q$~\cite{B}. The paper is devoted to the following question: 
\begin{itemize}
\item
{\em How can we characterize real Schur roots among all positive real roots of $\Delta$? In particular, given a positive real root, can we say whether it is a Schur root or not?}
\end{itemize}


In~\cite{Sch}, Schoefild provided a criterion for a real root being a Schur root in terms of the dimension vectors of subrepresentations. Hubery and Krause~\cite{HK} gave a characterization of real Schur roots in terms of non-crossing partitions. In~\cite{IS}, Igusa and Schiffler characterized Schur roots as those whose corresponding reflections are prefixes of a Coxeter element. In~\cite{LL}, K.-H.~Lee and K.~Lee suggested a (conjectural) method to explicitly identify real Schur roots, and proved it for acyclic quivers with multiple arrows between every pair of vertices (we call them {\em $2$-complete}) of rank $3$. We will give below our reformulation of their conjecture for $2$-complete quivers (though the equivalence may not be immediate).

Given an $n$-punctured disc $D$ with a base point, its fundamental group is a free group $F_n$ with $n$ generators $s_1,\dots,s_n$. Thus, to every element of the free group $F_n$ we can assign a loop in $D$. Now consider a quotient $W$ of $F_n$ by setting all the generators to be involutions. The quotient is the universal Coxeter group with $n$ generators, and it can be understood as the Weyl group of the root system $\Delta$ defined above (here we use the fact $Q$ is $2$-complete).  In particular, reflections of $W$ are in one-to-one correspondence with positive real roots of $\Delta$. Considering the canonical projection $F_n\to W$, we can assign to every element of $W$ a (finite) class of loops in $D$. Now take all the reflections in $W$ such that the assigned class of loops contains a representative without self-intersections. 

\begin{conj}[\cite{LL}]
The reflections in $W$ such that the assigned class of loops contains a representative without self-intersections are precisely those corresponding to real Schur roots in $\Delta$.
\end{conj}

We note that, given a positive real root $\alpha\in\Delta$, it is easy to check whether the assumptions of the conjecture hold for $\alpha$ by using the following reformulation (we refer to Section~\ref{r->arc} for the details), which makes the proposed characterization of real Schur roots very convenient to use:

\medskip

{\em Consider an $n$-punctured disc (where punctures are ordered) with a boundary marked point. Then  there is a one-to-one correspondence between Schur roots and non-self-intersecting arcs connecting the boundary marked point and one of the punctures.}

\medskip
The correspondence above is written explicitly in Section~\ref{r->arc}.  

\medskip


Our interest in the problem comes from cluster algebras (though the problem and the main result of the paper, Theorem~\ref{thm-LL}, is formulated without any relation to cluster algebras). In~\cite{BGZ}, Barot, Geiss and Zelevinsky defined a mutation of a ``symmetrization'' of a skew-symmetrizable exchange matrix (called a {\em quasi-Cartan companion}). Based on this, Seven~\cite{S2} defined {\em admissible} quasi-Cartan companions which have particularly nice properties, and proved that for any acyclic quiver $Q$ any sequence of mutations applied to the initial admissible quasi-Cartan companion results again in an admissible one. Using this, he showed that mutations of $Y$\!-\,seeds of the cluster algebra $\A(Q)$ with principal coefficients can be modeled by partial reflections of collections of roots of the root system $\Delta$ (we remind the construction in Section~\ref{s-refl}). We are interested in a natural question: which roots of $\Delta$ belong to some $Y$\!-\,seed? This question is actually equivalent to the question above about real Schur roots.

In the present paper, we prove the Lee -- Lee conjecture for $2$-complete quivers, and investigate the related combinatorics. We proceed according to the following plan. First, we investigate how do $Y$-seeds look on the Cayley graph of $W$ and how the mutations act on them. We collect the results in Figures~\ref{ac-fig}-\ref{ind2}. Next, we associate to every element of a $Y$-seed (i.e., to every {\em $\cc$-vector}) a curve on a certain hyperbolic orbifold homeomorphic to a sphere with one cusp and $n$ orbifold points of order $2$. We then pass from these to certain non-self-intersecting curves (we call them {\em arcs}) on an $n$-punctured disc $D$ which are in one-to-one correspondence with non-self-intersecting loops in $D$ corresponding to reflections of $W$, and prove (Theorem~\ref{disc}) that the set of real Schur roots can be embedded into the set of arcs, confirming the Lee -- Lee conjecture in one direction.

Next, we provide a combinatorial characterization of $Y$-seeds in terms of arcs in $D$. We define a notion of a {\em bad pair} of arcs in $D$ (which is immediate to verify, see Definition~\ref{def-bad} and Remark~\ref{bad}) and prove the following theorem.
  
\setcounter{section}{6}
\setcounter{theorem}{4}

\begin{theorem}
A clockwise ordered $n$-tuple of non-intersecting arcs  in $D$ corresponds to a $Y$\!-\,seed if and only it contains at most one bad pair.

\end{theorem}
One of the main tools in the proof is a theorem of Speyer and Thomas~\cite[Theorem 1.4]{ST} (also reproduced as Theorem~\ref{ST} below).

Finally, we show (Corollary~\ref{arc->inside_seed}) that every arc in $D$ can be included in an $n$-tuple of arcs with at most one bad pair. This implies that every arc corresponds to a real Schur root, so we get the following result.

\setcounter{theorem}{19}

\begin{theorem}
There is a natural one-to-one correspondence between real Schur roots and arcs in $D$. In particular, Lee -- Lee conjecture holds for $2$-complete quivers.

\end{theorem}
\setcounter{section}{1}

We would also like to emphasize the relations between our approach and the one used in~\cite{LL}. K.-H. Lee and K. Lee use the bijective correspondence between real Schur roots and $\dd$-vectors of non-initial cluster variables of the corresponding acyclic cluster algebra established by Caldero -- Keller~\cite{CK} and Caldero -- Zelevinsky~\cite{CZ}. We use a different characterization of real Schur roots as positive $\cc$-vectors of the corresponding cluster algebra -- this was established by Nagao~\cite{N} and N{\'a}jera Ch{\'a}vez~\cite{NCh}. This allows us to use the results of Speyer -- Thomas~\cite{ST} and Seven~\cite{S2} and their geometric interpretation in terms of partial reflections. 

One of the main tools in our considerations is the coincidence of two groups (see also Section~\ref{open}): both the Weyl group of the root system constructed by a $2$-complete acyclic quiver with $n$ vertices and the fundamental group of a hyperbolic orbifold of genus zero with one cusp and $n$ orbifold points of order $2$ are universal Coxeter groups of the same rank, and thus are isomorphic. This gives a particular embedding of the Cayley graph of the Weyl group into the hyperbolic plane and allows us to represent positive $\cc$-vectors by loops on the orbifold above (cf.~\cite{Bes}), and thus we can consider $Y$-seeds as $n$-tuples of loops. We note here that the Riemann surface used in~\cite{LL} is a double cover of our orbifold, so all our considerations and the main results can be formulated in terms of this double cover as well. However, the orbifold is more convenient for us as a step to introducing the model on the $n$-punctured disc, and we prefer to present our final criterion in terms of (collections of) curves on the punctured disc: this seems to us to be more explicit and easier to verify.       

\medskip

The paper is organized as follows. In Section~\ref{s-refl} we first remind essential details about quiver mutations, and then recall the geometric construction modelling the mutations of a mutation-acyclic quiver via a reflection group. In Section~\ref{cor} we list some immediate corollaries of the construction above, in particular ones concerning $2$-complete quivers; most of these are known, but there are also some we have not met in the literature. We then restrict ourselves to $2$-complete acyclic quivers. In Section~\ref{seeds} we consider $Y$-seeds drawn on the Cayley graph of the universal Coxeter group, and describe all the possible shapes of these (together with the combinatorics of mutations). The main results are collected in Figures~\ref{ac-fig}--\ref{ind2}. Section~\ref{sec-disc} is devoted to assigning of an $n$-tuple of non-intersecting arcs in an $n$-punctured disc $D$ to every $Y$-seed. In particular, this implies that we can assign an arc to every Schur root. Finally, in Section~\ref{converse} we characterize the $Y$-seeds in terms of collections of arcs in $D$, and complete the proof of (our reformulation of) Lee -- Lee conjecture. We also explain why our statement is equivalent to the initial conjecture in~\cite{LL}.

\subsection*{Acknowledgements}
We would like to thank Kyungyong Lee for very helpful discussions and comments to an earlier version of the paper, and Ralf Schiffler for suggesting the name {\em $2$-complete} for a quiver with all multiple arrows. We are grateful to the referee for valuable comments.

\section{Mutation-acyclic quivers via reflections}
\label{s-refl}

In this section we recall the construction from~\cite{rank3} showing that mutations of  a mutation-acyclic quiver can be modeled via a reflection group acting on some quadratic space. The results of this section can also be deduced from~\cite{S2,ST}, we just give a geometric interpretation. 

\subsection{Quiver mutations}
\label{s-def}
First, we remind the basics on quivers and their mutations.

A {\em quiver} $Q$ is a finite oriented graph with weighted edges containing no loops and no $2$-cycles, where weights are positive integers. We call the directed edges {\em arrows}, while drawing a quiver we omit weights equal to one. By {\em rank} of $Q$ we mean the number of its vertices.

For every vertex $k$ of a quiver $Q$ one can define an involutive operation  $\mu_k$ called {\it mutation of $Q$ in direction $k$}. This operation produces a new quiver  denoted by $\mu_k(Q)$ which can be obtained from $Q$ in the following way (see~\cite{FZ1}): 
\begin{itemize}
\item
orientations of all arrows incident to the vertex $k$ are reversed; 
\item
for every pair of vertices $(i,j)$ such that $Q$ contains arrows directed from $i$ to $k$ and from $k$ to $j$ the weight of the arrow joining $i$ and $j$ changes as described in Figure~\ref{quivermut}.
\end{itemize} 

\begin{figure}[!h]
\begin{center}
\psfrag{a}{\small $p$}
\psfrag{b}{\small $q$}
\psfrag{c}{\small $r$}
\psfrag{d}{\small $r'$}
\psfrag{k}{\small $k$}
\psfrag{mu}{\small $\mu_k$}
\epsfig{file=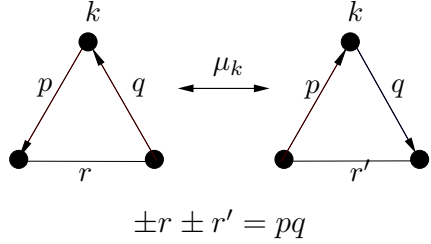,width=0.35\linewidth}\\
\medskip
$\pm{r}\pm{r'}={pq}$
\caption{Mutations of quivers. The sign before ${r}$ (resp., ${r'}$) is positive if the three vertices form an oriented cycle, and negative otherwise. Either $r$ or $r'$ may vanish. If $pq$ is equal to zero then neither the value of $r$ nor orientation of the corresponding arrow changes.}
\label{quivermut}

\end{center}
\end{figure}

Given a quiver $Q$, its {\it mutation class} is a set of all quivers obtained from $Q$ by all sequences of iterated mutations. Quivers from one mutation class are called {\it mutation-equivalent}.  

Quivers without loops and $2$-cycles are in one-to-one correspondence with integer skew-symmetric matrices $B=\{b_{ij}\}$, where $b_{ij}>0$ if and only if there is an arrow from $i$-th vertex to $j$-th one with weight $b_{ij}$. In terms of the matrix $B$ the mutation $\mu_k$ can be written as $\mu_k(B)=B'$, where
$$b'_{ij}=\left\{
           \begin{array}{ll}
             -b_{ij}, & \hbox{ if } i=k \hbox{ or } j=k; \\
             b_{ij}+\frac{|b_{ik}|b_{kj}+b_{ik}|b_{kj}|}{2}, & \hbox{ otherwise.}\\
           \end{array}
         \right.
$$
This transformation is called a {\it matrix mutation}. The skew-symmetric matrix $B$ is called an {\it exchange matrix} corresponding to $Q$. 

A quiver (and the corresponding exchange matrix) is called {\em acyclic} if it contains no oriented cycles. A quiver (and the matrix) is {\em mutation-acyclic} if its mutation class contains an acyclic representative. A quiver is $2$-{\it complete} if $|b_{ij}|\ge 2$  for every pair of distinct $i$ and $j$.


\subsection{Construction}
\subsubsection{The initial configuration.}
Let $Q$ be an acyclic quiver of rank $n$, and let $B$ be the corresponding skew-symmetric $n\times n$ matrix. We will assume that the vertices of $Q$ are indexed in a way such that $b_{ij}\ge 0$ for $i<j$. Consider a symmetric matrix with non-positive off-diagonal entries $M(B)=(m_{ij})$, where 
$$ m_{ii}=2, \qquad \quad  m_{ij}=-|b_{ij}| \ \text{ if } i\ne j. $$
This matrix (called a {\em Cartan companion of $B$, see~\cite{BGZ}}) defines a quadratic form (this is twice the {\it Tits quadratic form} of the path algebra of $Q$, see~\cite{B}), and we can consider $M(B)$ as a Gram matrix (i.e., the matrix of inner products) of some $n$-tuple of basis vectors $(v_1,\dots,v_n)$  in a quadratic $n$-space $V$ of the same signature as $M(B)$ has. 

\subsubsection{Reflection group and the root system.}
Given a vector $v\in V$  with $\langle v,v\rangle=2$ one can consider a {\it reflection } 
$$r_v(u)=u-\langle u,v\rangle v$$ 
with respect to $\Pi_v=v^\perp$, where $\langle\cdot,\cdot\rangle$ is the inner product defined by the quadratic form.
It is straightforward to see that $r_v$ preserves the quadratic form in $V$ 
and that $r_v(v)=-v$, i.e. that $r_v$ is a (pseudo-)orthogonal transformation preserving $\Pi_v$ and interchanging the halfspaces into which $V$ is decomposed by $\Pi_v$.

We denote by $G$ the group generated by reflections $s_1=r_{v_1},\dots,s_n=r_{v_n}$ in hyperplanes $\Pi_i=\Pi_{v_i}$. According to~\cite{V}, $G$ acts discretely in some cone $\C\subset V$ with fundamental chamber $F=\prod\limits_{i=1}^n\Pi_i^-$, where $\Pi_i^-=\{u\in V\ | \ \langle u,v_i\rangle<0  \}$. The fundamental chamber can also be understood as a connected component of the complement of the mirrors of all reflections inside $\C$. The images of vectors $v_i$ under $G$ are precisely real roots of the root system $\Delta$ constructed by the generalized Cartan matrix $M(B)$, the vectors $v_i$ are simple roots, $G$ is the corresponding Weyl group.  

Note that if $Q$ is $2$-complete, then the hyperplanes $\Pi_i$ (and their images under the action of $G$) do not intersect each other inside the cone $\C$, and the group $G$ is a free product of $n$ copies of $\Z/2\Z$.


\subsubsection{Roots and $\cc$-vectors.}

Let $Q$ be a quiver, and let $B$ be the corresponding $n\times n$ exchange matrix. Define an associated $2n\times n$ matrix $\t B$ in the following way: its top half is $B$, and the bottom half is the identity matrix. Mutating $\t B$ according to the rule above, we obtain a new matrix $\t B^t$ consisting on a top part $B^t$ (also called {\it principal part}) and the bottom part $C^t$, where $t$ is a sequence of mutations. The matrix $C^t$ is called a {\it $\cc$-matrix}, its columns are {\it $\cc$-vectors} $\cc_1^t,\dots,\cc_n^t$. Abusing notation, we will call a collection $S^t=(B^t,(\cc_1^t,\dots,\cc_n^t))$ (or $(Q^t,(\cc_1^t,\dots,\cc_n^t))$) a {\it $Y$\!-\,seed}.  

It was shown in~\cite{DWZ,ST} that if $Q$ is acyclic then for any sequence of mutations all $\cc$-vectors are roots of the root system $\Delta$ (where the coordinates are the coefficients in the basis consisting of the simple roots). Moreover, the following criterion for a collection of roots to form a $C$-matrix is proved in~\cite{ST}.

\begin{theorem}[\cite{ST}, Theorem 1.4]
\label{ST}
A collection of roots $u_1,\dots,u_n$ is the set of $\cc$-vectors for a $Y$\!-\,seed if and only if the following two assumptions hold:
\begin{enumerate}
\item If $u_i$ and $u_j$ are both positive roots or both negative roots then $\langle u_i,u_j\rangle\le 0$.
\item
It is possible to order the roots so that the positive roots precede the negative roots, and the product of the reflections corresponding to these roots, taken in this order, equals $s_1\dots s_n$. 
\end{enumerate}
\end{theorem} 

Furthermore, it was shown in~\cite{N} that all positive $\cc$-vectors are real Schur roots (i.e., dimension vectors of indecomposable rigid modules over a path algebra of $Q$). The converse inclusion was proved in~\cite{NCh}: $\cc$-vectors are precisely real Schur roots and their opposites.

\subsubsection{Mutation.}
The initial acyclic quiver $Q$ (and the initial matrix $B$) corresponds to the initial set of generating reflections $s_1,\dots,s_n$ in the group $G$ and to the initial domain $F\subset V$.
Applying mutations, we will obtain other sets of generating reflections in $G$ as well as other domains in $V$.

More precisely, define mutation of the set of initial generating reflections as a partial conjugation:
$$
\mu_k(s_j)=\begin{cases}  
s_ks_js_k & \text{if  $b_{jk}<0$,}\\ 
s_j&  \text{otherwise.}\end{cases}
$$
Consequently, the mutation of an $n$-tuple of simple roots (and of $n$-tuple of hyperplanes) is defined by a partial reflection:
$$
\mu_k(v_j)=\begin{cases}  
v_j-(v_j,v_k)v_k & \text{if $b_{jk}<0$,}\\ 
-v_k & \text{if $j=k$,}\\
v_j&  \text{otherwise.}\end{cases}
$$
In the general case, for a collection of roots $u_1,\dots,u_n$ and a skew-symmetric matrix $B^t=(b_{ij}^t)$ of signed inner products, the mutation $\mu_k$ is also defined by a partial reflection depending on the sign of the root $u_k$: 
$$
\mu_k(u_j)=\begin{cases}  
\text{if $u_k$ is positive, then }\; \begin{cases}u_j-(u_j,u_k)u_k & \text{if $b^t_{jk}<0$,}\\ 
-u_k & \text{if $j=k$,}\\
u_j&  \text{otherwise.}\end{cases}\\
\\
\text{if $u_k$ is negative, then } \begin{cases}u_j-(u_j,u_k)u_k & \text{if $b^t_{jk}>0$,}\\ 
-u_k & \text{if $j=k$,}\\
u_j&  \text{otherwise.}\end{cases}
\end{cases}
$$
Geometrically, mutation $\mu_k$ reflects (with respect to $u_k$) all the roots $u_j$ such that there is an arrow $k\to j$ (if $u_k$ is positive) or $j\to k$ (if $u_k$ is negative), takes $u_k$ to its negative, and leaves all the other roots intact.

One can note that if the reflections in roots $\{u_i\}$ generate the Weyl group $G$, then the reflections in $\{\mu_k(u_i)\}$ also generate $G$. This implies that after every mutation we will obtain a collection of generating reflections of $G$.

The following result by Seven shows that the  mutations of the quiver agree with mutations of the roots.

\begin{theorem}[\cite{S2}, Corollary 1.7]
\label{mut-trans}
Let $Q$ be an acyclic quiver of rank $n$, and let $B$ be the corresponding exchange matrix. Let $\V= \{v_1,\dots,v_n \}$, $v_i\in V$ be an $n$-tuple of vectors such that $\langle v_i,v_i\rangle=2$ and $\langle v_i,v_j\rangle=-|b_{ij}|$ for $i\ne j$. 
For a sequence of mutations $t=\mu_{i_s}\circ\dots\circ\mu_{i_1}$ 
denote $Q^t=\mu_{i_s}\circ\dots\circ\mu_{i_1}(Q)$, $B^t=\mu_{i_s}\circ\dots\circ\mu_{i_1}(B)$  and $\V^t=\mu_{i_s}\circ\dots\circ \mu_{i_1}(\V)$.
Then 
\begin{enumerate}
\item
$|\langle v_i^t,v_j^t\rangle|=|b_{ij}^t|$.
\item
$(B^t,(v_1^t,\dots,v_n^t))$ is a $Y$\!-\,seed.
\end{enumerate}
\end{theorem}
In other words, Theorem~\ref{mut-trans} says that mutating the initial configuration of roots by partial reflections we obtain $n$-tuples of $\cc$-vectors belonging to one $Y$\!-\,seed.





\section{Corollaries}
\label{cor}
In this section, we list some corollaries of the geometric construction above. Note that most of these follow from~\cite{W} where they are proved by purely combinatorial methods.

\begin{cor}
\label{min}
Let $Q$ be a $2$-complete acyclic quiver. Then 
\begin{itemize}
\item[(1)] All quivers in the mutation class of $Q$ are $2$-complete.

\item[(2)] If $Q'$ is a quiver in the mutation class of $Q$ and $\{b_{ij} \}$  and $\{b_{ij}' \}$ are the weights of arrows in $Q$ and $Q'$,
then $|b_{ij}|\le |b_{ij}'|$ for all $i\ne j$.

\end{itemize}

\end{cor}

\begin{cor}
\label{sep}
Let $Q'$ be a non-acyclic quiver mutation-equivalent to a $2$-complete acyclic quiver $Q$ of rank $n$. Then 
\begin{itemize}
\item[(1)] There exists an acyclic subquiver $Q_1$ of $Q'$ of rank $n-1$.

\item[(2)] Denote the vertex $Q'\setminus Q_1$ by $k$. Then the vertices of $Q_1$ split in two groups $I$ and $J$, such that for every $i\in I$ and $j\in J$ one has $b_{ik}>0$, $b_{kj}>0$, and $b_{ji}>0$.

\end{itemize}

\end{cor}

\begin{remark}
\label{geom}
Corollary~\ref{sep} was proved in~\cite{W}. The geometric meaning of the corollary is the following: given a non-acyclic $Y$\!-\,seed (and the corresponding $n$-tuple of hyperplanes in the cone $\C$), there is exactly one hyperplane separating others. Removing the vertex of $Q'$ corresponding to this hyperplane results in an acyclic subquiver.   
\end{remark}

\begin{definition}[Increasing/decreasing mutation]
\label{decr}
Let $Q$ be a quiver. We say that a mutation $\mu_k$ of $Q$ is {\it increasing}  if it increases the weight of at least one arrow and does not decrease all the other weights. Similarly, $\mu_k$ is {\it decreasing}  if it decreases the weight of at least one arrow and does not increase all the other weights.

\end{definition}

\begin{cor}
\label{decreasing}
Let $Q'$ be a non-acyclic quiver mutation-equivalent to a $2$-complete acyclic quiver $Q$. Then there exists a unique number $k\in \{1,\dots,n\}$ such that the mutation $\mu_k$ of $Q'$ is decreasing. All the other mutations $\mu_i$, $i\ne k$ are increasing.
 
\end{cor}

Corollary~\ref{decreasing} (also proved in~\cite{W}) follows from Corollary~\ref{sep}: the decreasing mutation corresponds to the only hyperplane separating others. This, in its turn, follows from the geometric interpretation of the weights of $Q'$ as inner products of the corresponding vectors (see Theorem~\ref{mut-trans}). The mutations that are neither increasing nor decreasing are sink/source mutations in acyclic $Y$\!-\,seeds.

\begin{remark}
Corollary~\ref{decreasing} implies that for every quiver $Q'$ mutation-equivalent to a $2$-complete acyclic quiver $Q$ there is an algorithm transforming $Q'$ to its acyclic representative: one only needs to apply the decreasing mutations finitely many times. The same procedure provides a finite time check whether a given quiver is mutation-equivalent to a $2$-complete acyclic one.

\end{remark}

The following corollary also follows from Remark~\ref{geom}. 

\begin{cor}
Let $Q$ be a $2$-complete acyclic quiver, and let
 $Q'$ be mutation-equivalent to $Q$. Suppose that $Q_1\subset Q'$ is a subquiver having a source (or sink).
Then $Q_1$ is acyclic.  

\end{cor}

\begin{definition}
Given a quiver $Q$ (and the exchange matrix $\t B$), we can define the corresponding {\em exchange graph} in the following way. Vertices correspond to $Y$\!-\,seeds $S^t=(B^t,(\cc_1^t,\dots,\cc_n^t))$ (carrying precisely the same information as the matrix $\t B^t$), and two vertices are joined by an edge if the corresponding $Y$\!-\,seeds can be obtained one from another by a single mutation.  

\end{definition}

Note that the exchange graph is always $n$-regular. The following statement is another immediate corollary of Corollary~\ref{decreasing}.

\begin{cor}
\label{tree-cor}
Let $Q$ be a $2$-complete acyclic quiver. Then the corresponding exchange graph  is an $n$-regular tree.
\end{cor}

Combinatorially, increasing mutations send a $Y$\!-\,seed further away from the ``line'' of acyclic $Y$\!-\,seeds (and from the initial $Y$\!-\,seed), while the decreasing one moves it towards the initial $Y$\!-\,seed.

\section{$Y$\!-\,seeds on the Cayley graph}
\label{seeds}
From now on, we deal with a $2$-complete acyclic quiver $Q$.
 
As we have seen before, $\cc$-vectors can be identified with roots of the root system $\Delta$ constructed by $Q$, so they correspond to some elements (reflections) of the Weyl group $G$. Since $Q$ is $2$-complete, $G$ is the universal Coxeter group
$$ G=\langle s_1,\dots,s_n \ | \ s_i^2=e \rangle,$$
so its Cayley graph $\G$ is an $n$-regular tree. The vertices of the Cayley graph are the elements of $G$, and the edges are labeled by $s_i$. We can assign reflections to the edges of $\G$: if an edge labeled by $s_i$ emanates from a vertex $w\in G$, then we can write the reflection $r_i=ws_iw^{-1}$ on this edge (one can easily check that this assignment is well-defined, i.e. it does not depend on the endpoint of the edge). Note that, since $G$ has no Coxeter relations of odd degree, the generators $s_i$ are not conjugate to each other, and thus the index $i$ in the presentation $ws_iw^{-1}$ is defined uniquely for every reflection. Moreover, since $\G$ is a tree, the reduced word for $w$ is also defined uniquely.   

Therefore, every $Y$\!-\,seed provides an $n$-tuple of reflections, which we can find on the Cayley graph $\G$. This section is devoted to understanding a general form of such $n$-tuples on $\G$ and the combinatorics of mutations of these $n$-tuples. This will be one of the key tools in Section~\ref{converse}.    

\subsection{Nodes}
\label{sec-nodes}

We will place a {\it node} at the midpoint of each edge of $\Ca$ (we use the term ``node'' to emphasize the difference with the vertices of $\Ca$). The union of all nodes decomposes the graph into $n$-star shaped ``fundamental domains'', we will call them {\em fundamental $n$-stars}. 

We will think about $\Ca$ as embedded into a plane, with rays in each star labeled by the corresponding reflections $r_1,r_2\dots,r_n$ in a clockwise order, see Fig.~\ref{tree}.

\begin{figure}[!h]
\begin{center}
\psfrag{a}{\small (a)}
\psfrag{b}{\small (b)}
\psfrag{s1}{\scriptsize $s_1$}
\psfrag{s2}{\scriptsize $s_2$}
\psfrag{s3}{\scriptsize $s_3$}
\psfrag{s4}{\scriptsize $s_4$}
\psfrag{s5}{\scriptsize $s_5$}
\psfrag{s2s1s2}{\scriptsize $s_2s_1s_2$}
\psfrag{s2s3s2}{\scriptsize $s_2s_3s_2$}
\psfrag{s2s4s2}{\scriptsize $s_2s_4s_2$}
\psfrag{s2s5s2}{\scriptsize $s_2s_5s_2$}
\psfrag{s3s1s3}{\tiny $s_3s_1s_3$}
\psfrag{s3s1s2s1s3}{\tiny $s_3s_1s_2s_1s_3$}
\psfrag{s3s1s3s1s3}{\tiny $s_3s_1s_3s_1s_3$}
\psfrag{s3s1s4s1s3}{\tiny $s_3s_1s_4s_1s_3$}
\psfrag{s3s1s5s1s3}{\tiny $s_3s_1s_5s_1s_3$}
\epsfig{file=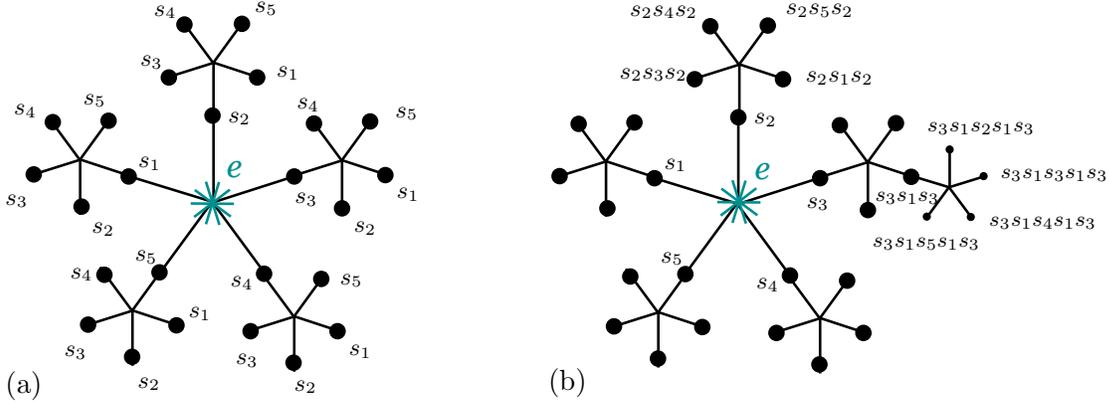,width=0.95\linewidth}
\caption{Part of the Cayley graph $\Ca$ for $G$ with the nodes in the midpoints of edges: (a) labeled by the generators,  (b) by associated reflections.}
\label{tree}
\end{center}
\end{figure}

\subsection{Ordered $Y$\!-\,seeds}
\label{sec order}

We now want to define a (natural) order on the $\cc$-vectors in a $Y$\!-\,seed. 

For $2$-complete acyclic quiver $Q$ the {\it natural order} on vertices of $Q$ is given by $i<j$ if $b_{ij}>0$ 
(i.e. if there is an arrow from $i$ to $j$). So that, the source is the minimal vertex and the  sink is the maximal one.

For a non-acyclic quiver $Q^t$ in a $2$-complete acyclic mutation class, let $\mu_k$ be the  decreasing mutation.
Recall from Section~\ref{cor} that the vertices of quiver $Q^t$ can be described as $I\cup \{k\}\cup J$, where  for every $i\in I$ and $j\in J$ one has $I=\{i \ | \ b_{ik}>0 \}$, $J=\{j \ | \ b_{jk}<0 \}$, 
so that the subquiver spanned by  $I$ and $J$ is acyclic, and for every $i\in  I$ and $j\in J$ one has $b_{ji}>0$. 
So, if we reverse the arrows between  $i\in  I$ and $j\in J$, the quiver $Q^t$ turns into a $2$-complete acyclic quiver (denote it $\t Q^t$).

The {\it natural order} on the vertices of a non-acyclic quiver  $Q^t$ is defined to be the natural order on  $\t Q^t$. 

By an {\it ordered $Y$\!-\,seed} we mean a $Y$\!-\,seed with the vertices of underlying quiver renumbered according to the natural order.

\subsection{Acyclic $Y$\!-\,seeds} 
\label{as}
The initial generating reflections $s_1,\dots,s_n$ are associated with the nodes lying in one fundamental $n$-star (labeled by $e$).
It is also easy to see that any tuple of generating reflections $r_1,\dots,r_n$ obtained from any acyclic $Y$\!-\,seed looks similarly:

\begin{prop} 
\label{acyclic}
If $r_1,\dots, r_n$ are reflections arising from an acyclic $Y$\!-\,seed, then the corresponding nodes of $\G$ lie in one fundamental $n$-star.

\end{prop}

\begin{proof}
Any acyclic $Y$\!-\,seed can be obtained from the initial $Y$\!-\,seed via a sequence of sink/source mutations. Each sink/source mutation either 
preserves all generating reflections or conjugates all of them. In both cases they still lie in one fundamental $n$-star (either the same or not). 

\end{proof}

Knowing the  reflections arising from an acyclic $Y$\!-\,seed is not enough to know the $Y$\!-\,seed itself. What we need in addition are the signs of all roots and the order (the weights of the quiver are known since they are identical to those of $Q$). 

\begin{figure}[!h]
\begin{center}
\psfrag{i}{\small \it Initial $Y$\!-\,seed}
\psfrag{s1}{\scriptsize $s_1$}
\psfrag{s2}{\scriptsize $s_2$}
\psfrag{s3}{\scriptsize $s_3$}
\psfrag{s4}{\scriptsize $s_4$}
\psfrag{s5}{\scriptsize $s_5$}
\psfrag{si}{\tiny sink}
\psfrag{so}{\tiny source}
\epsfig{file=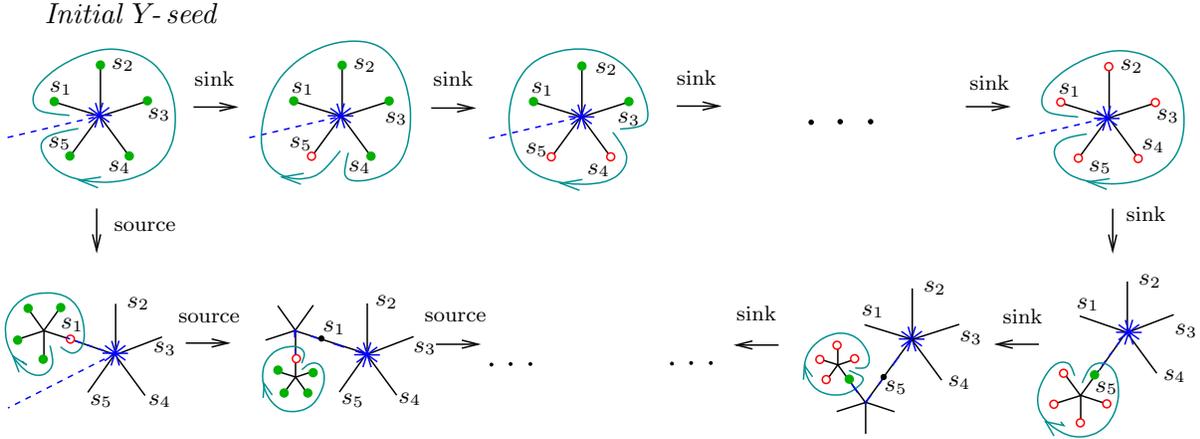,width=0.99\linewidth}
\caption{ Acyclic $Y$\!-\,seeds.}
\label{ac-fig}
\end{center}
\end{figure}

All the ways how a $Y$\!-\,seed can look like are shown in Figure~\ref{ac-fig}. We mark the positive roots as (green) filled nodes, and negative roots as (red) unfilled nodes (the coloring originates from the notion of {\em maximal green sequence}, see~\cite{K}); we will use red/green notation instead of unfilled/filled in the sequel. The natural order is shown by a long arrow around the $Y$\!-\,seed: it is clockwise with the source at the beginning of the arrow.
Figure~\ref{ac-fig} can be easily verified by using the definition of a $Y$\!-\,seed mutation.

An easy induction shows that a product of the nodes (i.e., the corresponding reflections) in clockwise order starting from the first green is precisely a Coxeter element $s_1\dots s_n$ of $G$. For the $Y$\!-\,seeds corresponding to the initial fundamental chamber (the top row in Fig.~\ref{ac-fig}) we draw a dotted ray from $e$ to indicate the first node in the product. Note that this agrees with Theorem~\ref{ST}.

\subsection{Separating nodes}

By a {\it shortest} path between two points on $\Ca$ we mean a geodesic path with respect to the graph metric  (taking in account that $\Ca$ is a tree, a ``shortest'' path is the same as a ``non-returning'' path).

Given nodes $N_1,\dots,N_n$ on $\Ca$, we define a {\it convex hull} of these nodes as a union of all shortest paths 
between $N_i$ and $N_j$, $i,j\in \{1,\dots n \}$. It is a finite subtree of $\Ca$ (if both vertices and nodes are considered).

We say that a node $N$ {\it separates} nodes $N_i$ and $N_j$ in $\Ca$ if the shortest path connecting $N_i$ with $N_j$ contains $N$.
Given a set of nodes $N_1,\dots,N_n$, we call $N_i$ a {\it separating node} if $N_i$ separates at least two other nodes (i.e., $N_i$ belongs to the convex hull of all the other nodes). 

Using Proposition~\ref{acyclic}, Remark~\ref{geom} can be now reformulated in the following way.

\begin{prop}
\label{sep-node}
Let $Q$ be a $2$-complete acyclic quiver. Let $(Q^t,(\cc_1^t,\dots,\cc_n^t))$ be a $Y$\!-\,seed, and let $N_1,\dots,N_n$ be the corresponding nodes of $\G$. If $Q^t$ is acyclic then no $N_i$ is separating, otherwise exactly one of $\{N_i\}$ is a separating node.  

\end{prop}

\subsection{Non-acyclic $Y$\!-\,seeds}
\label{nas}
We are now ready to describe sets of nodes from non-acyclic $Y$\!-\,seeds. 

We draw the $Y$\!-\,seeds schematically in the following way. Consider the convex hull $\T$ of the nodes in a non-acyclic $Y$\!-\,seed. $\T$ is a finite tree (we count as vertices both nodes and vertices of $\G$) with $n-1$ leaves (all of them are nodes) containing one more node (of valence two) lying in the chosen seed (the separating node). All nodes have a color, green or red. Also, there is a distinguished vertex of $\T$ which is the closest to the vertex $e\in \G$, we draw the path from $e$ to the tree as a dotted segment. 

\begin{figure}[!h]
\begin{center}
\psfrag{1}{\tiny $1$}
\psfrag{2}{\tiny $2$}
\psfrag{n}{\tiny $n$}
\psfrag{m}{\tiny $m$}
\psfrag{m+1}{\tiny $m\!\!+\!\!1$}
\psfrag{m-1}{\tiny $m\!\!-\!\!1$}
\psfrag{q}{\tiny $q$}
\psfrag{q+1}{\tiny $q\!\!+\!\!1$}
\epsfig{file=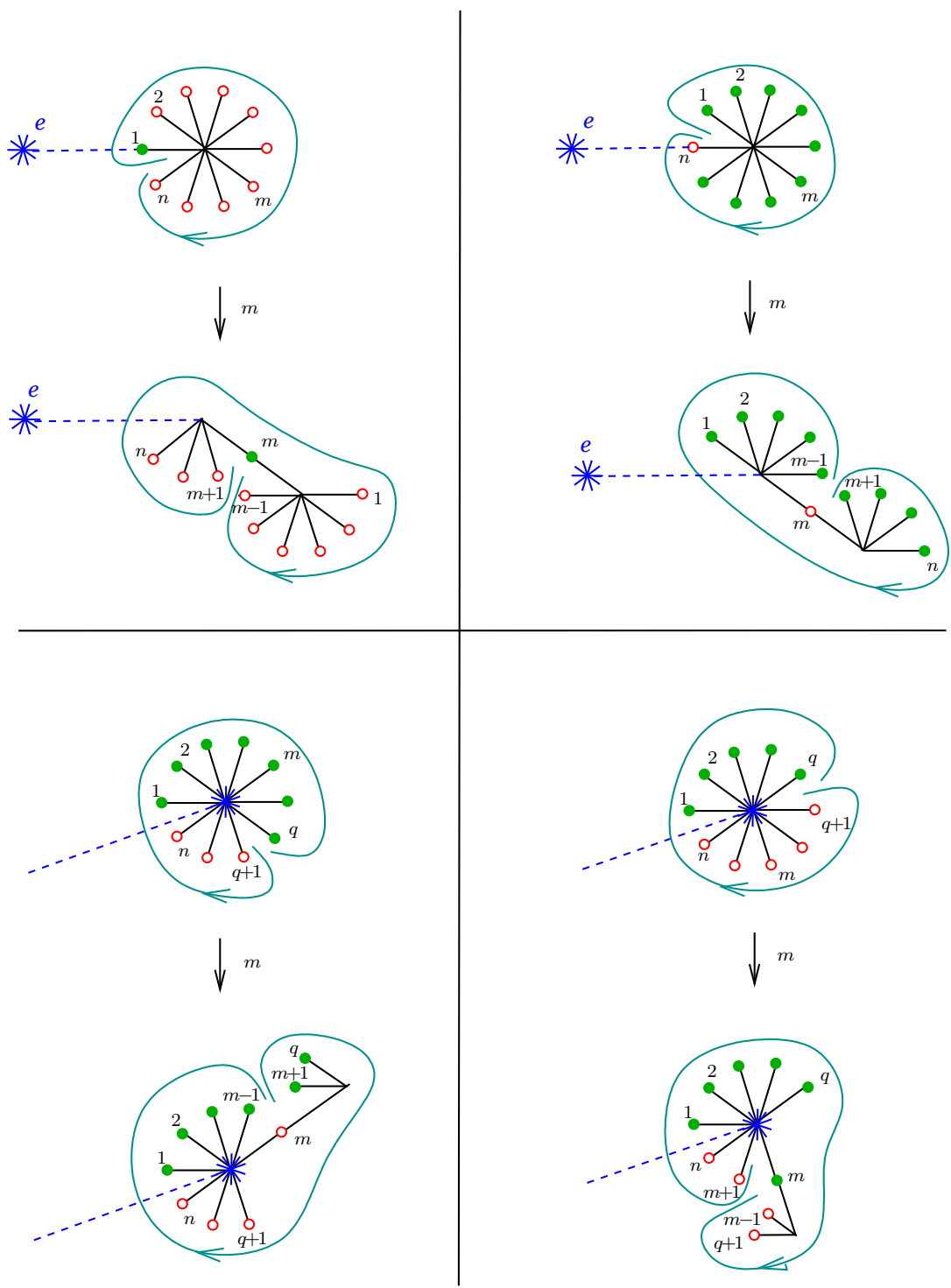,width=0.9\linewidth}
\caption{Base of the induction.}
\label{base}
\end{center}
\end{figure}

The $Y$\!-\,seeds that can be obtained by one mutation from an acyclic $Y$\!-\,seed are shown in Fig.~\ref{base}. According to Section~\ref{as}, there are exactly three types of acyclic $Y$\!-\,seeds, for all of them possible mutations are shown in Fig.~\ref{base} depending on the mutated vertex $m$. Two of the types are shown in the top row, here $m$ can be any vertex except for $1$ and $n$ (which are source and sink respectively). The third type of an acyclic $Y$\!-\,seed is shown in the bottom row. Here $m$ can be any vertex except for $q$ and $q+1$ (which are source and sink respectively), the left and right pictures correspond to the different possible colors of $m$ (i.e., green on the left and red on the right). 

\begin{figure}[!h]
\begin{center}
\psfrag{e}{\scriptsize \color{blue} $e$}
\psfrag{1}{\tiny $1$}
\psfrag{2}{\tiny $2$}
\psfrag{n}{\tiny $n$}
\psfrag{k}{\tiny $k$}
\psfrag{k+1}{\tiny $k\!\!+\!\!1$}
\psfrag{k-1}{\tiny $k\!\!-\!\!1$}
\psfrag{m}{\tiny $m$}
\psfrag{m1}{\tiny $1\le m\le q$}
\psfrag{m2}{\tiny $q+1\le m\le k-1$}
\psfrag{m3}{\tiny $k+1\le m\le n$}
\psfrag{m+1}{\tiny $m\!\!+\!\!1$}
\psfrag{m-1}{\tiny $m\!\!-\!\!1$}
\psfrag{q}{\tiny $q$}
\psfrag{q+1}{\tiny $q\!\!+\!\!1$}
\epsfig{file=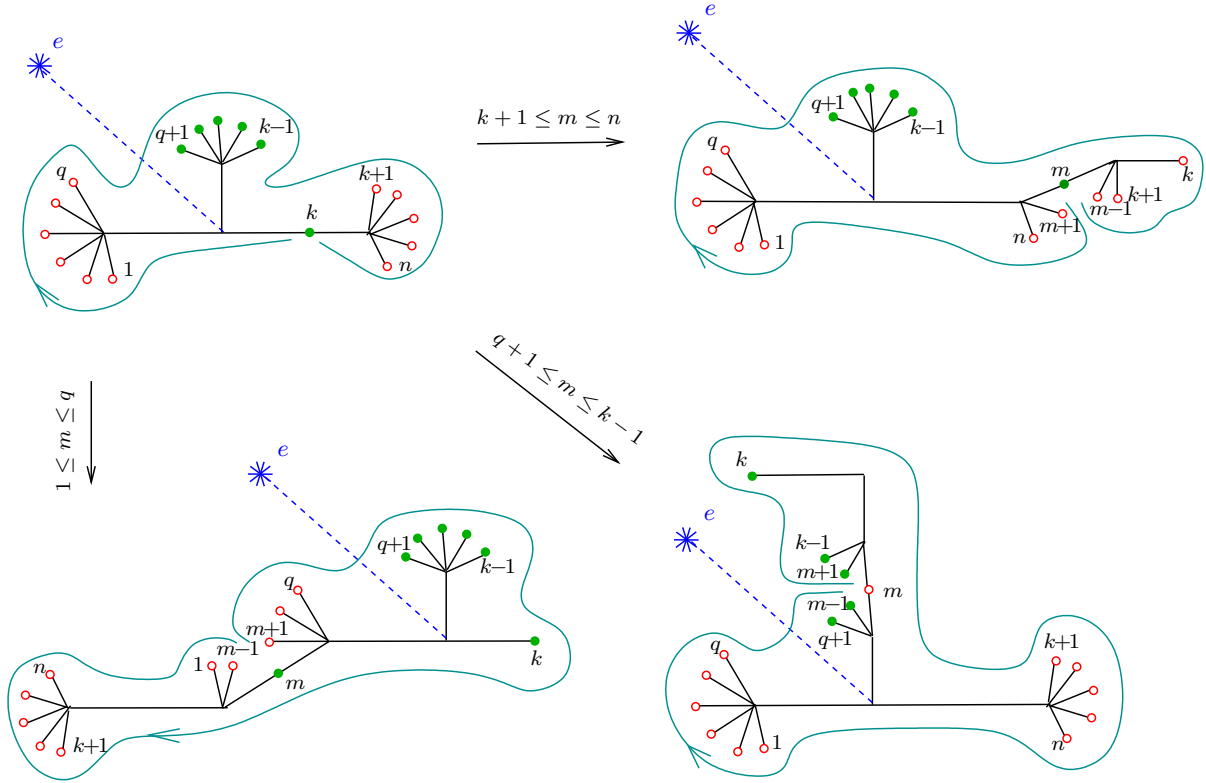,width=0.999\linewidth}
\caption{Inductive step 1: the case of positive separating node. Three mutations $\mu_m$ are shown, depending on the group to which the vertex $m$ belongs. 
 }
\label{ind1}
\end{center}
\end{figure}

\begin{figure}[!h]
\begin{center}
\psfrag{e}{\scriptsize \color{blue} $e$}
\psfrag{1}{\tiny $1$}
\psfrag{2}{\tiny $2$}
\psfrag{n}{\tiny $n$}
\psfrag{k}{\tiny $k$}
\psfrag{k+1}{\tiny $k\!\!+\!\!1$}
\psfrag{k-1}{\tiny $k\!\!-\!\!1$}
\psfrag{m}{\tiny $m$}
\psfrag{m1}{\tiny $1\le m\le q$}
\psfrag{m2}{\tiny $q+1\le m\le k-1$}
\psfrag{m3}{\tiny $k+1\le m\le n$}
\psfrag{m+1}{\tiny $m\!\!+\!\!1$}
\psfrag{m-1}{\tiny $m\!\!-\!\!1$}
\psfrag{q}{\tiny $q$}
\psfrag{q+1}{\tiny $q\!\!+\!\!1$}
\epsfig{file=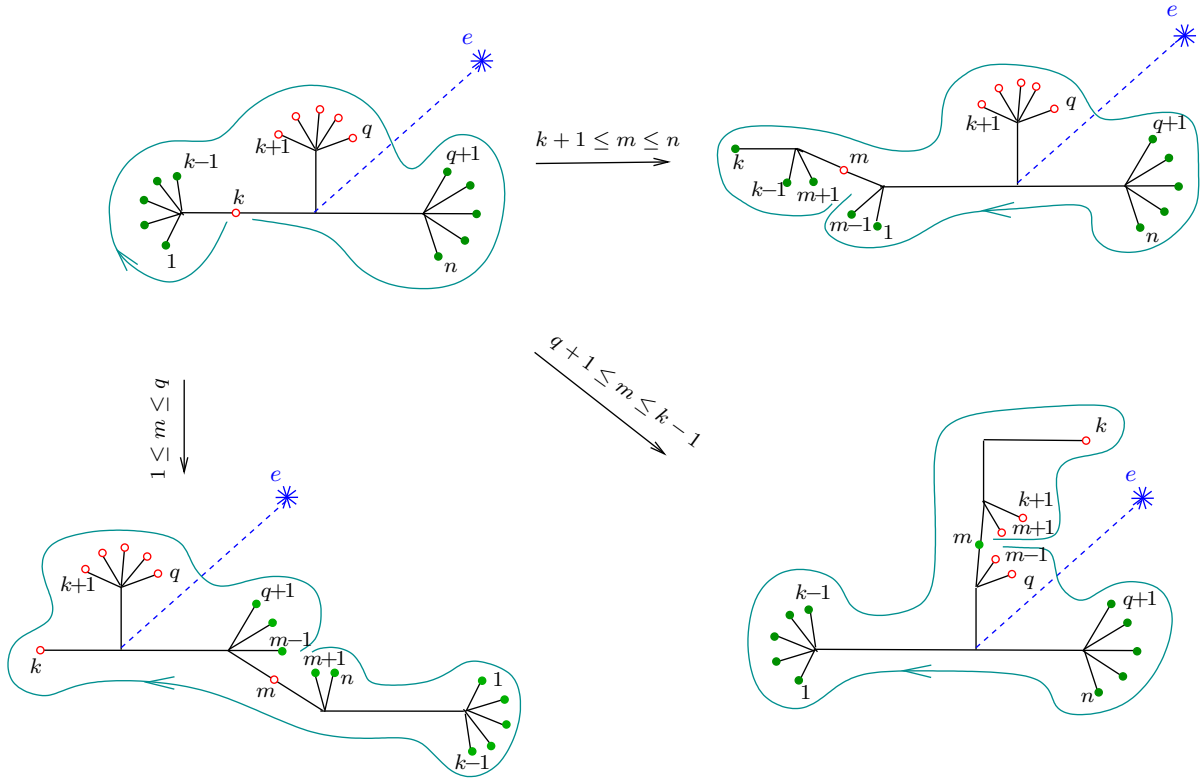,width=0.995\linewidth}
\caption{Inductive step 2: the case of negative separating node. Three mutations $\mu_m$ are shown, depending on the group to which the vertex $m$ belongs. 
}
\label{ind2}
\end{center}
\end{figure}

The general form of a non-acyclic $Y$\!-\,seed is shown in the top left diagrams in Figures~\ref{ind1} and~\ref{ind2}. These diagrams are obtained from the tree $\T$ by collecting together consequtive non-separating nodes of the same color. For example, the top left diagram in Fig.~\ref{ind1} means the following:
\begin{enumerate}
\item
Clockwise order on $\T$ with the first node marked $1$ and the last node marked $n$ (shown by the long curved arrow around the tree) coincides with the natural order.
\item
In the ordered $Y$\!-\,seed, there is at most three consequtive groups of nodes of the same color, with $q$ red nodes coming first, $k-q-1$ green nodes coming next, then a green separating node, and finally $n-k$ red nodes. Note that some of the groups may be empty (i.e., any one or even two of the three numbers above may vanish).
\item
The product of the reflections in a clockwise order starting from the first positive root (i.e., the first green node w.r.t. the natural order) is equal to a Coxeter element $s_1\dots s_n$.    
\end{enumerate} 

\begin{figure}[!h]
\begin{center}
\psfrag{e}{\scriptsize \color{blue} $e$}
\psfrag{1}{\tiny $1$}
\psfrag{2}{\tiny $2$}
\psfrag{n}{\tiny $n$}
\psfrag{k}{\tiny $k$}
\psfrag{k+1}{\tiny $k\!\!+\!\!1$}
\psfrag{k-1}{\tiny $k\!\!-\!\!1$}
\psfrag{m}{\tiny $m$}
\psfrag{m1}{\tiny $1\le m\le q$}
\psfrag{m2}{\tiny $q+1\le m\le k-1$}
\psfrag{m3}{\tiny $k+1\le m\le n$}
\psfrag{m+1}{\tiny $m\!\!+\!\!1$}
\psfrag{m-1}{\tiny $m\!\!-\!\!1$}
\psfrag{q}{\tiny $q$}
\psfrag{q+1}{\tiny $q\!\!+\!\!1$}
\psfrag{=}{\huge $=$}
\epsfig{file=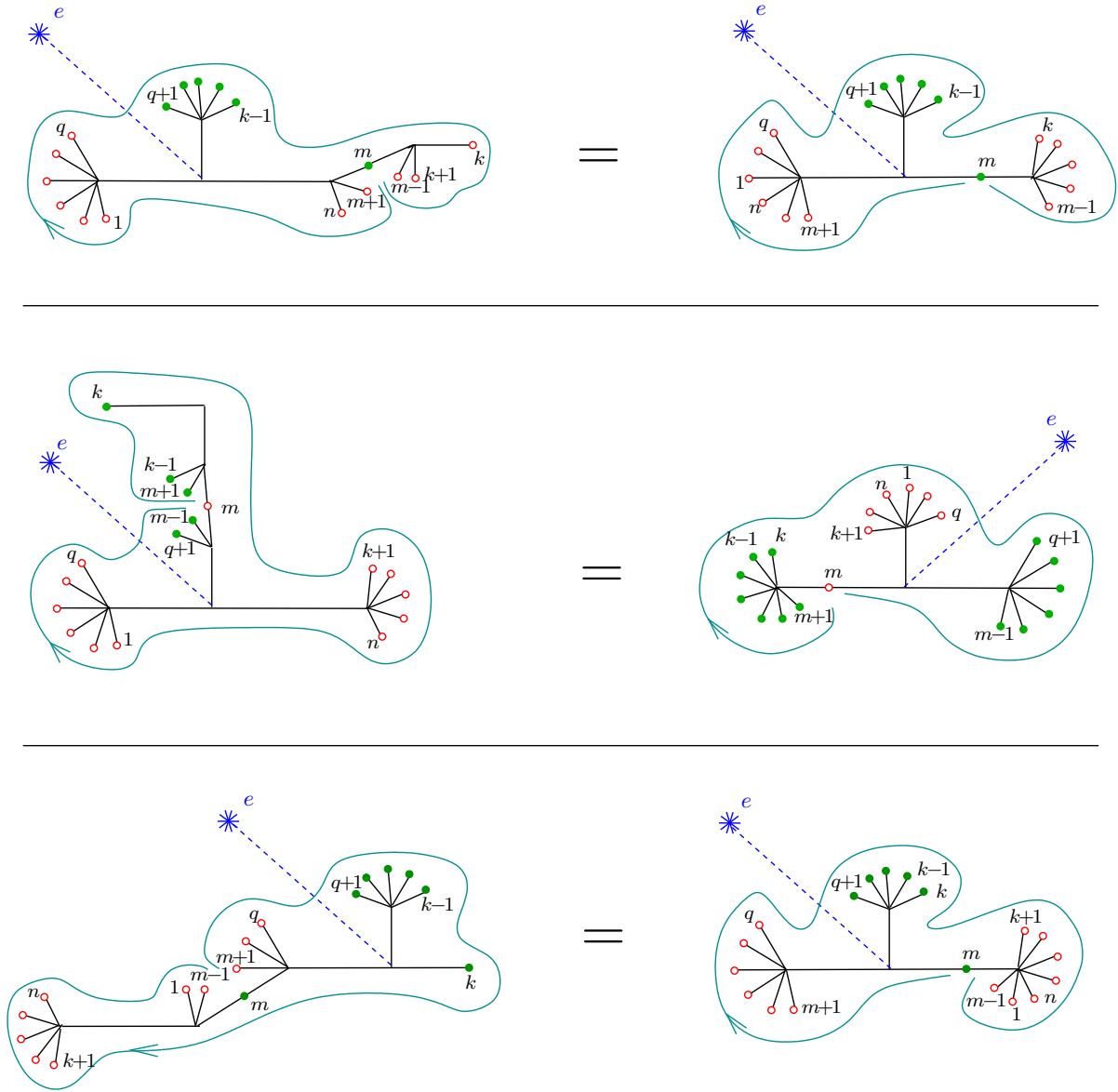,width=0.995\linewidth}
\caption{Equivalent trees}
\label{eq}
\end{center}
\end{figure}

To verify the fact that there are only two types of $Y$\!-\,seeds we use a straightforward induction by the number of mutations required to obtain the $Y$\!-\,seed. All the $Y$\!-\,seeds obtained by one mutation of an acyclic $Y$\!-\,seed are of these types, see Fig.~\ref{base}, this is the base of the induction. Then Figures~\ref{ind1} and~\ref{ind2} contain all possible mutations depending on the color of the mutated node (note that the mutation in the separating node is a decreasing one, so we do not need to consider it), and Figure~\ref{eq} shows that the obtained $Y$\!-\,seeds are also of one of the same two types (modulo appropriate renumbering of the nodes). The fact that the clockwise product of reflections (starting from the first node after the ray from $e$) is the required Coxeter element also follows by induction.

Note that this provides another (purely combinatorial) proof of the ``only if'' part of Theorem~\ref{ST}.

\section{$\cc$-vectors as arcs in a disc}
\label{sec-disc}
Let $D$ be a disc with a set of $n$ interior marked points $P=\{p_1, \dots,p_n\}$ and one boundary marked point $b$ called a {\it basepoint}. An {\it arc} on $D$ is a non-self-intersecting path in the interior of $D\setminus P$ with endpoints in $b$ and one of $p_i$. The aim of this section is to assign an arc to every real Schur root in $\Delta$, and a collection of non-intersecting arcs to every $Y$\!-\,seed.

\subsection{Universal Coxeter group and orbifold}

Let $\H$ be the hyperbolic plane (for example, we can take the Poincar\'e disc model). Consider a regular ideal $n$-gon $\P$ in $\H$ with sides $\{l_i\}$ indexed in a clockwise order, and choose a Euclidean ``midpoint'' $\p_i$ on every side $l_i$ of $\P$. Now consider the group $\Gamma$ generated by $n$ rotations $\rho_i$ by $\pi$ in points $\p_1,\dots,\p_n$. Applying the Poincar\'e's Fundamental Polyhedron Theorem (see e.g.~\cite{M}), we see that the group $\Gamma$ is a universal Coxeter group $(\Z/2\Z)^{*n}$ (with a non-standard representation: the generators $\rho_i$ act on $\H$ as rotations), and $\P$ is its fundamental domain. The hyperbolic plane $\H$ is tessellated by copies of $\P$, all the copies are indexed by the elements of $\Gamma$.

The Cayley graph of $\Gamma$ is an $n$-regular tree dual to the tessellation. Now observe that the Cayley graph is precisely the graph $\G$ described in Section~\ref{seeds} with an embedding in $\H$ (and thus in $\R^2$) described in Section~\ref{sec-nodes}. This corresponds to a natural isomorphism between groups $G$ and $\Gamma$ taking the generating reflections $s_i$ to the rotations $\rho_i$ around $\p_i$. From now on, we will identify groups $G$ and $\Gamma$ (and $s_i$ with $\rho_i$).

To every reflection $r$ in $G$ we assign two classes $[\bm\gamma'_r]$ and $[\hat{\bm\gamma}'_r]$ of non-self-intersecting paths on $\H$ in the following way. A path $\hat\gamma'_r\in [\hat{\bm\gamma}'_r]$ is any path connecting the center $O'$ of the $n$-gon $\P$ with the center $r(O')$ of its copy $r(\P)$ corresponding to the element $r\in\G$, avoiding the images of points $\p_i$ under the action of $G$.  A path $\gamma'_r\in [\bm\gamma'_r]$ is any non-self-intersecting path connecting the center of the $n$-gon $\P$ with the node of $\G$ corresponding to $r$, also avoiding the images of points $\p_i$. 

Now consider the quotient $\O$ of $\H$ by the action of the group $G$, denote by $\pi$ the canonical projection. The space $\O$ is a hyperbolic orbifold homeomorphic to a sphere with one cusp and $n$ orbifold points $\{\pi(\p_i)\}$ of order two. All the nodes of $\G$ project to orbifold points, where every orbifold point corresponds to precisely one conjugacy class of reflections. The path $\pi(\gamma'_r)$ connects $\pi(O')$ with one of the orbifold points, while the path $\pi(\hat\gamma'_r)$ is a loop from $\pi(O')$ going around $\pi(\gamma'_r)$. 

The following lemma easily follows from the construction above.

\begin{lemma}
\label{distinct}
\begin{enumerate}
\item
Given a reflection $r\in G$, a path $\pi(\gamma'_{r})\subset\O$ is defined uniquely up to homotopy. 
\item
Two paths $\pi(\gamma'_{r_1})$ and $\pi(\gamma'_{r_2})$ with $r_1\ne r_2$ are not homotopic to each other.
\end{enumerate}
\end{lemma}

\begin{proof}
Up to homotopy in $\H\setminus G(\{\p_i\})$, a path $\gamma'_r$ (or $\hat\gamma'_r$) is completely defined by the binary choices of the ``sides'' of intersection with the images $G(\{l_i\})$ of all the sides of the polygon $\P$: a path can intersect a side either on the left of the image of the corresponding $\p_i$ or on the right. However, these two halves are identified in $\O$, which proves the first assertion.

The second assertion is straightforward: the loops $\pi(\hat\gamma'_{r_1})$ and  $\pi(\hat\gamma'_{r_2})$ define distinct elements of the fundamental group of $\O$. 

\end{proof}

For simplicity, we will denote  $\pi(\gamma'_r)$ and  $\pi(\hat\gamma'_r)$  by  $\gamma_r$ and  $\hat\gamma_r$ respectively, and $\pi(O')$ by $O$.  

We can assume that every path $\gamma_r$ has minimal number of self-intersections within its homotopy class. For the same purpose, we can also assume that a loop $\hat\gamma_r$ is a small deformation of $\gamma^{-1}_r\gamma_r$, i.e. $\hat\gamma_r$ bounds a small neighborhood of $\gamma_r$ in $\O$. This defines $\gamma_r$ and $\hat\gamma_r$ up to isotopy.

A priori, it is not clear whether paths $\gamma_r$ are self-intersecting or not. We will find this out in the next section. 

\subsection{From orbifold to disc}

Given a geodesic on $\O$ connecting a point $\p\in\O$ and the cusp, we are not able to compute its length as it is infinite. However, we can {\it compare} lengths of two geodesics from $\p$ to the cusp in the following way. Choose a horocycle around the cusp bounding a small horoball not containing any of the points $\p_i$. Then we can measure the length of a geodesic between $\p$ and the horocycle, and compare these measurements for different geodesics. The order we get is independent of the choice of horocycle: changing a horocycle adds a constant (equal to the distance between the horocycles) to all the lengths.      

\begin{lemma}
\label{seed->arcs}
Let $S^t$ be an ordered $Y$\!-\,seed, and let $(r_1,\dots,r_n)$ be the corresponding $n$-tuple of reflections. Then it is possible to choose the paths $\gamma_{r_i}$ in their isotopy classes satisfying the following properties.
\begin{enumerate}
\item
All paths $\gamma_{r_i}$ neither are self-intersecting nor intersect each other except for the common endpoint $O$.  
\item
The paths $\gamma_{r_i}$ emanate from $O$ in a clockwise order.
\item
Loops $\hat\gamma_{r_i}$ and their inverses generate the fundamental group of $\O$.
\item
Denote by $\ell$ the shortest geodesic between the cusp and $O$ amongst those emanating from $O$ between $\gamma_{s_n}$ and $\gamma_{s_1}$. Then no of $\gamma_{r_i}$ intersects $\ell$. 
\item
The product of reflections $r_i$ in clockwise order starting from the first one after $\ell$ is equal to a Coxeter element $s_1\dots s_n$.
\end{enumerate}
\end{lemma}

\begin{proof}
First, note that the first two assertions of the lemma are equivalent to ones with all $\gamma_{r_i}$ substituted by $\hat\gamma_{r_i}$: every $\hat\gamma_{r_i}$ is a loop around $\gamma_{r_i}$, so they have precisely the same intersection properties. 

The proof is by induction on the number of mutations from the initial $Y$\!-\,seed $S$ to $S^t$. The initial $Y$\!-\,seed (given by the simple roots and the generating reflections $s_1,\dots,s_n$ on the tree) corresponds to a collection of paths on $\O$ satisfying all the assertions of the lemma: these are just geodesics from $O$ to $\p_i$.


Now suppose that a $Y$\!-\,seed $S^t$ is represented by a collection of non-intersecting paths on $\O$, and let $\mu_m(S^t)$ be a $Y$\!-\,seed obtained from $S^t$ by a mutation $\mu_m$.
Looking at all possible mutations listed in Figures~\ref{ac-fig}--\ref{ind2},  we see that each mutation corresponds to a transformation of the roots which can be described in the following way: several consequtive roots $\cc_{l_1},\dots,\cc_{m-1}$ or $\cc_{m+1},\dots,\cc_{l_2}$ are reflected with respect to $\cc_m$. In terms of reflections in $G$, this means that the corresponding consequtive reflections are conjugated by $r_m$. Now, in terms of the loops $\hat\gamma_{r_i}$, this means that the corresponding loops are also conjugated by $\hat\gamma_{r_m}$ (in particular, the third assertion follows immediately).

This transformation takes a non-intersecting collection of loops to a non-intersecting one (see Fig.~\ref{mut} for an example), which proves the first assertion. The second and the forth assertions follow from the explicit form of mutations shown in Figures~\ref{ac-fig}--\ref{ind2} (the geodesic $\ell$ corresponds to a continuation of the dashed ray in the figures). The last assertion now follows from the second one.  

\end{proof}

\begin{figure}[!h]
\begin{center}
\psfrag{m}{\small $\mu_3$}
\psfrag{l}{\small $\ell$}
\psfrag{1}{\scriptsize $\hat\gamma_{r_1} $}
\psfrag{2}{\scriptsize $\hat\gamma_{r_2} $}
\psfrag{3}{\scriptsize $\hat\gamma_{r_3} $}
\psfrag{4}{\scriptsize $\hat\gamma_{r_4} $}
\psfrag{1'}{\scriptsize $\hat\gamma_{r_1}' $}
\psfrag{2'}{\scriptsize $\hat\gamma_{r_2}' $}
\psfrag{3'}{\scriptsize $\hat\gamma_{r_3}' $}
\psfrag{4'}{\scriptsize $\hat\gamma_{r_4}' $}
\epsfig{file=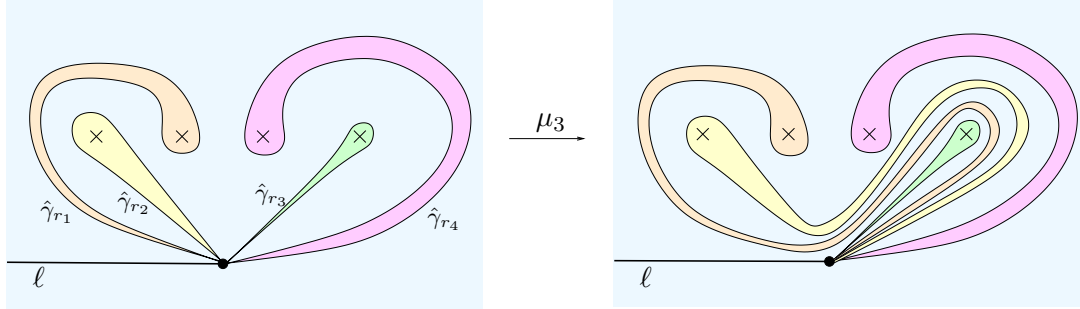,width=0.9\linewidth}
\caption{Mutation as a conjugation of loops}
\label{mut}
\end{center}
\end{figure}

Now, let us cut the orbifold $\O$ along $\ell$ and forget about the hyperbolic structure of $\O$. This results in a disc with $n$ interior marked points and two marked points at the boundary ($O$ and the cusp on $\O$). We can identify this with the disc $D$ defined in the beginning of Section~\ref{sec-disc}, where the basepoint $b=O$ and $p_i=\p_i$. We call the second (cuspidal) boundary marked point $\infty$.  

In view of the results of~\cite{N,NCh}, the following theorem is a straightforward corollary of Lemmas~\ref{distinct} and~\ref{seed->arcs}.

\begin{theorem}
\label{disc}
\begin{enumerate}
\item
Every real Schur root defines an arc in $D$. The arc is defined uniquely up to isotopy, distinct roots define distinct arcs.
\item
Every $Y$\!-\,seed defines an $n$-tuple of non-intersecting arcs. 
\item
The product of corresponding reflections in clockwise order is equal to a Coxeter element $s_1\dots s_n$.
\end{enumerate}
\end{theorem}

\section{Arcs in a disc as $\cc$-vectors}
\label{converse}
In Theorem~\ref{disc}, we showed that every real Schur root defines an arc in $D$, and every $Y$\!-\,seed defines an $n$-tuple of non-intersecting arcs. The aim of this section is to understand additional properties of such $n$-tuples, and to show that every arc belongs to at least one of them. This will imply that there is a bijection between real Schur roots and arcs in $D$.

\subsection{A reflection from an arc}
\label{r->arc}

We will draw the disc as an upper half-plane $\{z\in \CC \ | \ \Im\, z>0 \}$, with the points $p_i$ placed from left to right on the horizontal line $\Im\, z=1$ and the point $\infty$ at infinity.


Given any arc in $D$ we can construct a reflection in $G$ as follows (cf.~\cite[Section 3]{Bes}).
Denote by $\ell_i$ a vertical ray $\ell_i=\{z\ |\ \Re\, z=\Re\, p_i, \Im\, z>\Im\, p_i\}$. Let $\gamma$ be an arc ending at $p_k$. Going along $\gamma$ from $b$ to $p_k$, we list all the indices $i_1,i_2,\dots,i_l$ of rays  $\ell_i$ intersected by $\gamma$. Then we assign to $\gamma$ a reflection $r=s_{i_1}\dots s_{i_l}\,s_{k}\,s_{i_l}\dots s_{i_1}$, see Fig.~\ref{arc-ex} for an example. We can assume that the word for $r$ we obtained is reduced, i.e. no two consequtive letters are the same (this corresponds to $\gamma$ not intersecting the same $\ell_i$ twice in a row which can be easily achieved by an isotopy -- here we use the fact $\gamma$ does not have self-intersections).

\begin{figure}[!h]
\begin{center}
\psfrag{g}{\small \color{OliveGreen} $\gamma$}
\psfrag{1}{$\gamma=\gamma_r$, $r=s_2s_3s_2$}
\psfrag{2}{$\gamma=\gamma_r$, $r=s_3s_1s_2s_3\, s_4\, s_3s_2s_1s_3$}
\psfrag{l1}{{\small $\ell_1$}}
\psfrag{l2}{\small {$\ell_2$}}
\psfrag{l3}{\small {$\ell_3$}}
\psfrag{l4}{\small {$\ell_4$}}
\epsfig{file=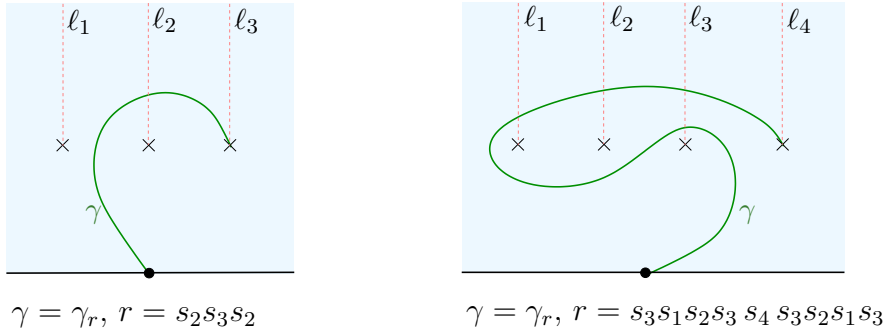,width=0.7\linewidth}
\caption{Reflection assigned to an arc: examples}
\label{arc-ex}
\end{center}
\end{figure}

\begin{lemma}
\label{arc->root} 
Let $r\in G$ be a reflection constructed by an arc $\gamma$. Then $\gamma=\gamma_r$.
\end{lemma}

\begin{proof}
The rays $\ell_i$ connect the points $p_i$ to $\infty$, so their preimages on the orbifold $\O$ are precisely $\pi(l_i)$, where $l_i$ are the sides of the initial fundamental polygon $\P\subset\H$. Thus, if we lift $\gamma$ to $\H$, an intersection with $\ell_i$ corresponds to moving to a neighboring fundamental domain of $G$ along the edge of the Cayley graph labeled by $s_i$, which coincides with the definition of $\gamma_r$ for a given reflection $r\in G$.

\end{proof}

\subsection{Bad pairs}

\begin{definition}
Define the following {\it partial order} on the reflections in $G$: for the reflections $r=ws_iw^{-1}$ and $r'=w's_j(w')^{-1}$ (written in the reduced form) we say that $r<r'$ if $w'=ws_iu$ for some $u\in G$, $|s_iu|=|u|+1$ (where $|w|$ denotes the length of the reduced representative of $w$ in $G$). 
We extend this definition to a {\it partial order on the arcs in $D$} as follows:  given two arcs $\gamma=\gamma_r$ and $\gamma'=\gamma_{r'}$, we say that 
$\gamma<\gamma'$ if $r<r'$.

\end{definition}

\begin{definition}
\label{def-bad}
Let $(\gamma_{r_1},\dots,\gamma_{r_k})$ be a $k$-tuple of non-intersecting arcs indexed in a clockwise order (w.r.t. the endpoint $b$). We say that two consequtive arcs $(\gamma_{r_i}, \gamma_{r_{i+1}})$ form a {\it bad pair} if either $\gamma_{r_i}<\gamma_{r_{i+1}}$ or $\gamma_{r_{i+1}}<\gamma_{r_i}$.

\end{definition}

\begin{remark}
\label{bad}
Geometrically, $r<r'$ means that the node $r$ separates the node $r'$ from $e$ in the Cayley graph. For the corresponding arcs, $\gamma_{r}<\gamma_{r'}$ means that some initial segment of  $\gamma_{r'}$ coincides with $\gamma_{r}$, after which $\gamma_{r'}$ intersects the vertical ray $\ell_i$ from the endpoint $p_i$ of $\gamma_{r}$.

\end{remark}

\begin{lemma}
\label{seed->1jump} 
Let $S$ be a $Y$\!-\,seed. Then the $n$-tuple of corresponding arcs (in the clockwise order) contains at most one bad pair. 

\end{lemma}

\begin{proof}
We have observed in Section~\ref{seeds} that positive $\cc$-vectors precede negative $\cc$-vectors in a $Y$\!-\,seed in the clockwise order.
We will prove that neither two positive nor two negative roots $\alpha$ and $\alpha'$ in a $Y$\!-\,seed give rise to a bad pair, which will show that there is at most one bad pair at the place where the sign changes. 

Let $\alpha$ and $\alpha'$ be two positive roots or two negative roots with corresponding reflections $r$ and $r'$. Suppose that $r$ and $r'$ give rise to a bad pair $(\gamma_{r},\gamma_{r'})$, we can assume that $r<r'$. This implies that the node $r$ separates the node $r'$ from $e$ in the Cayley graph. However, Figures~\ref{ac-fig}--\ref{ind2} show that this is not the case: a red node can separate $e$ from green ones only, and a green one can separate $e$ from red ones only. Thus, we come to a contradiction, which proves the lemma.

\end{proof}

\begin{lemma}
\label{1jump->seed}
Let   $(\gamma_{1},\dots,\gamma_{n})$ be a clockwise ordered non-intersecting $n$-tuple of arcs in $D$ containing at most one bad pair. Then it corresponds to a $Y$\!-\,seed. 

\end{lemma}

To prove the lemma, we need the following statement.

\begin{prop}
\label{*}
Let $\alpha, \alpha'\in \Delta$  be either two positive or two negative roots, denote the corresponding reflections by $r$ and $r'$ respectively. Assume that $\gamma_{r}$ and $\gamma_{r'}$ are arcs in $D$.
Then $(\gamma_{r},\gamma_{r'})$ is a bad pair if and only if $\langle \alpha,\alpha' \rangle >0$.

\end{prop}

\begin{proof}
The statement is an immediate corollary of the definition of a bad pair. The assumption $r<r'$ is equivalent to the fact that the node $r$ separates the node $r'$ from $e$ in the Cayley graph, which is equivalent to the fact that the hyperplane $\alpha^\perp$ separates the hyperplane  ${\alpha'}^\perp$ from the initial fundamental chamber in the cone $\C$, and this implies the statement of the proposition.

\end{proof}

\begin{proof}[Proof of Lemma~\ref{1jump->seed}]
According to Lemma~\ref{arc->root}, every arc $\gamma_i$ corresponds to a reflection $r_i\in G$, and thus to a root in $\Delta$ (defined up to a sign). To prove the lemma, we verify the conditions (1) and (2) of Theorem~\ref{ST}. 

We can choose the signs of the roots as follows: if there is no bad pair, then we take all roots to be positive, and if $(\gamma_k,\gamma_{k+1})$ is a bad pair, then we take positive roots corresponding to reflections $r_{1},\dots, r_{k}$, and negative roots corresponding to reflections $r_{k+1},\dots, r_{n}$. According to Prop.~\ref{*}, the condition (1) of Theorem~\ref{ST} is then satisfied. Moreover, to verify condition (2) we are only left to show that  $r_1\dots r_n=s_1\dots s_n$.  

The group of automorphisms of  $D$ is the braid group $\mathbf B_n$ (see e.g. Theorem 1.10 in~\cite{Bir}). In particular, by the action of $\mathbf B_n$ we can take  the $n$-tuple $(\gamma_{1},\dots,\gamma_{n})$ to  $(\gamma_{s_1},\dots,\gamma_{s_n})$ (see~\cite[Section 4]{Bes} for an algorithm), for which the required equality holds: the standard generators of $\mathbf B_n$ act by exchanging endpoints of arcs (i.e., we take a closed curve containing the endpoints of $\gamma_i$ and $\gamma_{j}$ which intersects $\gamma_i$ and $\gamma_{j}$ once and does not intersect any other arc, and apply a ``half-twist'' in this curve, see~\cite{Bir} for more details). The standard computation below (see e.g.~\cite{Bes}) shows that the action of the generators of $\mathbf B_n$ does not change the product of the reflections, which implies the lemma.

Indeed, the transformation above (more precisely, one of the two mutually inverse ones) acts on the reflections $r_k$ in the following way (we assume $i<j$): it takes $r_i$ to $(r_jr_{j-1}\dots r_{i+1})r_i(r_{i+1}\dots r_{j-1}r_j)$, takes $r_j$ to $(r_{i+1}\dots r_{j-1})r_j(r_{j-1}\dots r_{i+1})$, and exchanges these two (i.e., it agrees with the {\em braid operations} defined in~\cite{CB}), see Examples~\ref{ex-conj} and~\ref{ex1-conj}. As a result, the subword $r_{i}\dots r_{j}$ in the product is now substituted by 
$$
(r_{i+1}\dots r_{j-1})r_j(r_{j-1}\dots r_{i+1})\,r_{i+1}\dots r_{j-1}\,(r_jr_{j-1}\dots r_{i+1})r_i(r_{i+1}\dots r_{j-1}r_j)=r_{i}\dots r_{j},
$$
while the other parts of the product remain intact. Thus, the product of all reflections also remains intact. 

\end{proof}

\begin{example}
\label{ex-conj}
Take a triple of arcs shown in Fig.~\ref{braid}(a) corresponding to reflections
$$r_1=s_1s_3s_1,\quad r_2=s_1,\quad r_3=s_3s_2s_3,$$ and consider the action of the standard generator $\sigma_2\in\mathbf B_3$ exchanging the last two marked points counter-clockwise (see Fig.~\ref{braid}(b)). This corresponds to a half-twist in the closed curve shown in Fig.~\ref{braid}(c), the result is shown in   Fig.~\ref{braid}(d). Now, the new triple of reflections is 
$$r_1'= s_1s_3s_2s_3s_1=r_2r_3r_2,\quad r_2'=s_1=r_2,\quad r_3'=s_3s_2s_3s_2s_3=r_3r_2r_1r_2r_3,   
$$ 
so $r_1'r_2'r_3'=(r_2r_3r_2)r_2(r_3r_2r_1r_2r_3)=r_1r_2r_3\,(=s_1s_2s_3)$.

\begin{figure}[!h]
\begin{center}
\psfrag{a}{\small (a)}
\psfrag{b}{\small (b)}
\psfrag{c}{\small (c)}
\psfrag{d}{\small (d)}
\psfrag{s2}{\small $\sigma_2$}
\psfrag{r1}{\scriptsize $r_1$}
\psfrag{r2}{\scriptsize $r_2$}
\psfrag{r3}{\scriptsize $r_3$}
\psfrag{r1'}{\scriptsize $r_1'$}
\psfrag{r2'}{\scriptsize $r_2'$}
\psfrag{r3'}{\scriptsize $r_3'$}
\epsfig{file=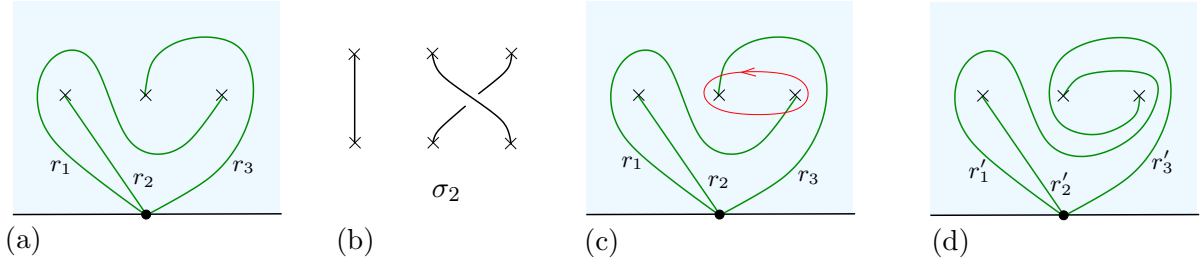,width=0.995\linewidth}
\caption{Action of a standard generator $\sigma_2\in\mathbf B_3$ on a triple of arcs}
\label{braid} 
\end{center}
\end{figure}
\end{example}

\begin{example}
\label{ex1-conj}
Consider the action of the braid $\sigma_1\sigma_2^{-1}\sigma_1^{-1}\in\mathbf B_3$ exchanging the first and the third marked points clockwise (see Fig.~\ref{braids}(b)) on the initial set of reflections $r_1=s_1,r_2=s_2,r_3=s_3$ (see Fig.~\ref{braids}(a)). This corresponds to a half-twist in the closed curve shown in Fig.~\ref{braids}(c), the result is shown in   Fig.~\ref{braids}(d). As in the previous example, we can compute the new triple of reflections: 
$$r_1'= s_1s_2s_3s_2s_1=r_1r_2r_3r_2r_1,\quad r_2'=s_2=r_2,\quad r_3'=s_2s_1s_2=r_2r_1r_2,   
$$ 
so $r_1'r_2'r_3'=r_1r_2r_3=s_1s_2s_3$.

\begin{figure}[!h]
\begin{center}
\psfrag{a}{\small (a)}
\psfrag{b}{\small (b)}
\psfrag{c}{\small (c)}
\psfrag{d}{\small (d)}
\psfrag{s}{\small $\sigma_1\sigma_2^{-1}\sigma_1^{-1}$}
\psfrag{r1}{\scriptsize $r_1$}
\psfrag{r2}{\scriptsize $r_2$}
\psfrag{r3}{\scriptsize $r_3$}
\psfrag{r1'}{\scriptsize $r_1'$}
\psfrag{r2'}{\scriptsize $r_2'$}
\psfrag{r3'}{\scriptsize $r_3'$}
\epsfig{file=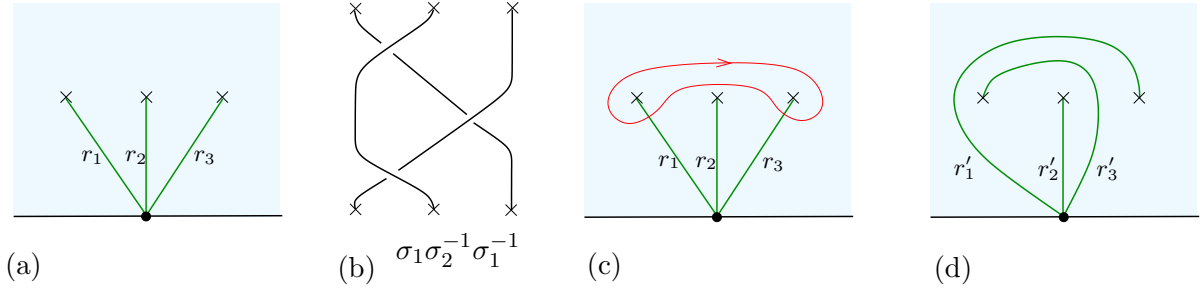,width=0.995\linewidth}
\caption{Action of a braid $\sigma_1\sigma_2^{-1}\sigma_1^{-1}\in\mathbf B_3$ on the initial triple of arcs}
\label{braids} 
\end{center}
\end{figure}

\end{example}

\begin{remark}
Note that an $n$-tuple of arcs may not define a $Y$\!-\,seed uniquely. For example, an $n$-tuple with no bad pair corresponds to $n+1$ distinct $Y$\!-\,seeds: these correspond to $k$ simple roots $v_1,\dots,v_k$ and $n-k$ negatives of simple roots $v_{k+1},\dots,v_n$, see Fig.~\ref{ac-fig}.    
 
\end{remark}

Lemmas~\ref{seed->1jump} and~\ref{1jump->seed} lead to the following theorem.

\begin{theorem}
\label{seed<->1jump}
A clockwise ordered $n$-tuple of non-intersecting arcs  in $D$ corresponds to a $Y$\!-\,seed if and only it contains at most one bad pair.

\end{theorem}

\subsection{Lee -- Lee conjecture}
\label{LL}
The aim of this section is to show that every arc in $D$ belongs to some $n$-tuple with at most one bad pair. This will complete the proof of the fact that the assignment $r\to\gamma_r$ defines a bijection between (reflections corresponding to) real Schur roots and arcs in $D$. We then show that this is equivalent to Lee -- Lee conjecture for $2$-complete quivers.

\begin{definition}[$\gamma$-twin arcs]
\label{twins}
Given an arc $\gamma\subset D$, two arcs $\beta_1$ and $\beta_2$ are {\it $\gamma$-twin arcs} if they are obtained as follows:
\begin{itemize}
\item[-] take a loop $\hat \gamma$ around $\gamma$ and a point $q\in\hat \gamma$;
\item[-] denote by $\s_1$ and $\s_2$ the two distinct paths along $\hat \gamma$ from the basepoint $b$ to $q$;
\item[-] connect $q$ with a marked point $p$ distinct from the endpoint of $\gamma$ by any non-self-intersecting path $\rho$ not intersecting $\hat\gamma\setminus q$;
\item[-] define $\beta_1= \rho\circ \s_1$, $\beta_2= \rho\circ \s_2$, see Fig.~\ref{twin}.
\end{itemize}

\begin{remark}
\label{twins-def}
Alternatively, twin arcs can be defined in the following way. Let $\gamma$ and  $\beta_1$ be two non-intersecting arcs, denote the corresponding reflections by $r_\gamma$ and $r_1=w_1 s_jw_1^{-1}$. Then $\beta_2$ is the arc corresponding to the reflection $r_2=w_2 s_jw_2^{-1}$, where $w_2^{-1}w_1=r_\gamma$, or equivalently  $w_2=w_1r_\gamma$.  
\end{remark}

\end{definition}

Note that two $\gamma$-twin arcs have representatives in their isotopy classes disjoint from each other and from $\gamma$ and $\hat \gamma$.

\begin{figure}[!h]
\begin{center}
\psfrag{g}{ \small $\gamma$}
\psfrag{g'}{\color{red} \small  $\hat \gamma$}
\psfrag{b1}{\color{magenta}  \small $\beta_1$}
\psfrag{b2}{\color{blue}  \small $\beta_2$}
\psfrag{s}{\scriptsize \color{DarkOrchid} $\rho$}
\epsfig{file=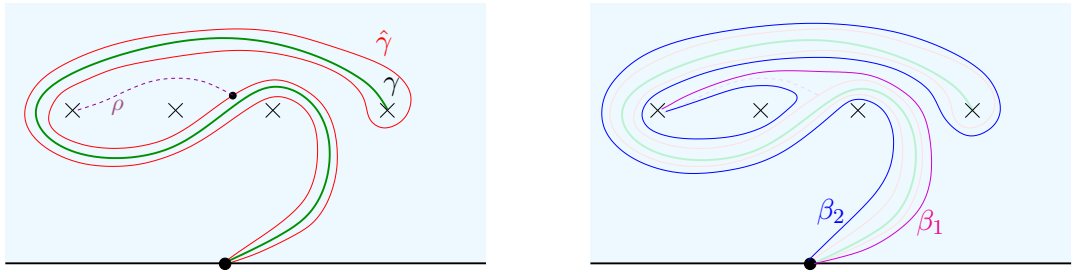,width=0.89\linewidth}
\caption{Construction of $\gamma$-twin paths $\beta_1$ and $\beta_2$.}
\label{twin}
\end{center}
\end{figure}

\begin{definition}
By the {\it length} of an arc $\gamma_r\subset D$ we mean the length of the reduced word for $r\in G$. We denote the length of $\gamma$ by $|\gamma|$.

\end{definition}

\begin{lemma}
\label{twin->nojump}
Let $\beta_1$ and $\beta_2$ be two $\gamma$-twin arcs with $|\gamma|<|\beta_i|$, $i=1,2$.
Then at least one of the pairs $(\beta_1,\gamma)$  and  $(\beta_2,\gamma)$ is not bad.

\end{lemma}

\begin{proof}

Suppose that both  $(\beta_1,\gamma)$  and  $(\beta_2,\gamma)$ are bad pairs. As  $|\gamma|<|\beta_i|$, $i=1,2$, this implies that  $\beta_1>\gamma$ and $\beta_2>\gamma$. Let $p_j$ be the endpoint of $\gamma$ distinct from $b$, recall that $\ell_j$ is the vertical ray from $p_j$ (see Fig.~\ref{good twin}). Since  
 $\beta_1>\gamma$ and $\beta_2>\gamma$, both $\beta_i$ first follow $\gamma$ (on different sides) and then cross the vertical ray $\ell_j$ (see Remark~\ref{bad}), we will assume that $\beta_2$ crosses $\ell_j$ closer to $p_j$. As $\beta_2$ has its endpoint at some marked point $p_k$ distinct from $p_j$, the arc $\beta_2$ should leave the disc bounded by $\gamma$, $\ell_j$ and $\beta_1$ (shaded disc  in Fig.~\ref{good twin}).
However, $\beta_2$ cannot cross neither $\gamma$ nor $\beta_2$, which implies that it will leave the disc through $\ell_j$. Hence, there exists a representative of $\beta_2$ in its isotopy class which does not enter the shaded disc at all, which contradicts the assumption that $\beta_2>\gamma$.
The contradiction proves the statement.

\end{proof}

\begin{figure}[!h]
\begin{center}
\psfrag{g}{\color{OliveGreen} \small $\gamma$}
\psfrag{b1}{\color{blue}  \small $\beta_1$}
\psfrag{b2}{\color{magenta}  \small $\beta_2$}
\psfrag{t}{\scriptsize  $\ell_j$}
\psfrag{p}{  $p_j$}
\epsfig{file=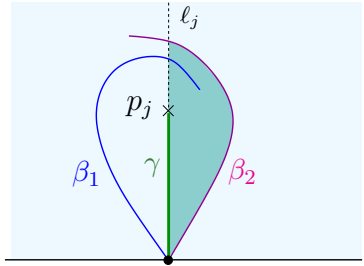,width=0.3\linewidth}
\caption{If $\beta_1>\gamma$ and $\beta_2>\gamma$ then  one of $(\beta_1,\gamma)$ and $(\beta_2,\gamma)$ is not bad.  }
\label{good twin}
\end{center}
\end{figure}

\begin{remark}
\label{length_of_twin}
By construction, if $\beta_1$ and $\beta_2$ are $\gamma$-twin arcs then $\big| |\beta_1|-|\beta_2| \big |\le |\gamma|$.

\end{remark}









The following proposition is an immediate corollary of Remark~\ref{twins-def}.
\begin{prop}
\label{twin_exists}
Given two arcs $\gamma,\delta\subset D$ with distinct endpoints, there exists an arc $\beta$ such that $\beta$ and $\delta$ are $\gamma$-twin. 

\end{prop}



\begin{prop}
\label{n-1_without_jump}
Any arc in $D$ can be included into a clockwise ordered $(n-1)$-tuple of non-intersecting arcs $(\gamma_1,\dots,\gamma_{n-1})$ without bad pairs.

\end{prop}

\begin{proof}
The proof is by induction on the number of arcs $n$. The base case is $n=2$; it holds trivially as an arc $\gamma$ is always contained in the set $\{\gamma\}$ of $1$ arc without bad pairs. We will assume that the statement holds for a disc with $n-1$ marked points and show it for the disc with $n$ marked points.

Suppose that $\gamma$ has $p_k$ as an endpoint. Let $p_i$, $i\ne k$ be any other  marked point. Remove the marked point $p_i$ from $D$ and use inductive assumption to choose a set of $n-2$ arcs $\gamma_1,\dots,\gamma_{n-2}$ including $\gamma$ and having  no bad pairs (in the disc with $n-1$ marked points). Now, we put the marked point $p_i$ back to $D$ and observe that no pair of arcs can become bad from this: indeed, removing a marked point is equivalent to taking a quotient by one of the generators, and if two words are not contained in each other in the quotient then they are definitely not contained in each other in the bigger group.

Our next aim is to add an  $(n-1)$-st arc $\beta$ to $\gamma_1,\dots,\gamma_{n-2}$ such that no bad pair arise.
Let $p_i$ and $p_j$ be the marked points distinct from the endpoints of  $\gamma_1,\dots,\gamma_{n-2}$.
Consider a closed non-self-intersecting curve $\s$ separating  $p_i$ and $p_j$ from all the other marked points and not intersecting any of $\gamma_1,\dots, \gamma_{n-2}$. 

First, we take any arc $\beta_0$ preceding $\gamma_1$ in the clockwise order and connecting $b$ to $p_i$ (it does exist as the arcs $\gamma_1,\dots, \gamma_{n-2}$ do not separate any domain in $D$). 
Applying the Dehn twist $T_\s$ (or its inverse) along $\s$ to the arc $\beta$ several times if needed, we can assume that 
\begin{equation}
\label{eqn} 
|\beta_0|>n|\gamma_i|  \ \  \text{for} \ \  i=1,\dots,n-2.
\end{equation}
If  $(\beta_0,\gamma_1)$ is not a bad pair then we are done.
Assume that $(\beta_0,\gamma)$ is a bad pair. Let $\beta_1$ be a $\gamma_1$-twin of  $\beta_0$, it exists by Proposition~\ref{twin_exists}.
By Lemma~\ref{twin->nojump} and inequality~(\ref{eqn}) the pair $(\beta_1,\gamma_1)$ is not bad. If $(\beta_1,\gamma_2)$ is not a bad pair we  are done again, otherwise consider $\beta_2$ defined as a $\gamma_2$-twin of $\beta_1$, etc (we can continue using Lemma~\ref{twin->nojump} due to inequality~(\ref{eqn}) together with Remark~\ref{length_of_twin}). Considering, if needed, $\beta_i$ (defined as $\gamma_i$-twin of $\beta_{i-1}$) for all $i$, we either find $\beta_i$ which does not form a bad pair neither with $\gamma_i$ nor with $\gamma_{i+1}$ for $i\le n-1$, or observe that $(\gamma_{n-2},\beta_{n-2}) $ cannot be a bad pair by construction (and there is no right neighbor of $\beta_{n-2}$).    


\end{proof}

\begin{cor}
\label{arc->inside_seed}
Let $\gamma$ be an arc in $D$. Then there exists an $n$-tuple of non-intersecting arcs containing $\gamma$ with at most one bad pair.

\end{cor}

\begin{proof}
By Proposition~\ref{n-1_without_jump}, $\gamma$ can be included into a clockwise ordered $(n-1)$-tuple $(\gamma_1,\dots,\gamma_{n-1})$ of non-intersecting arcs in $D$ without bad pairs. 
Let $p_i$ be the marked point which is not an endpoint of  any of $\gamma_1,\dots,\gamma_{n-1}$. Connect $b$ to $p_i$ by the arc $\gamma_0$ preceding $\gamma_1$ in the clockwise order, so that the $n$-tuple $(\gamma_0,\dots,\gamma_{n-1})$ is clockwise ordered. Since adding $\gamma_0$ cannot introduce more than one bad pair, we get the statement. 

\end{proof}

Combining Corollary~\ref{arc->inside_seed} with Theorems~\ref{disc} and~\ref{seed<->1jump}, we obtain the following theorem.

\begin{theorem}[Lee -- Lee conjecture for $2$-complete quivers]
\label{thm-LL}
For a $2$-complete acyclic quiver, the assignment $r\to\gamma_r$ defines a bijection between (reflections corresponding to) real Schur roots and arcs in $D$.

\end{theorem}

We are left to explain why Theorem~\ref{thm-LL} is equivalent to Lee -- Lee conjecture~\cite[Conjecture 2.4]{LL}  for $2$-complete quivers. We refer to~\cite[Section 2.1]{LL} for notation.  

Consider an index two subgroup $\Gamma_0\triangleleft\Gamma$ containing all the words of even length. Then the quotient $\H/\Gamma_0$ is precisely the manifold $\Sigma_\sigma$ from~\cite[Section 2.1]{LL}, with segments $L_1$ and $L_2$ projecting to $\ell$ under the two-fold covering of $\O$. 

Now perform the following isotopy of any $\sigma$-admissible curve $\eta$ on $\Sigma_\sigma$ (see~\cite[Definition 2.1]{LL}) defined in a small neighborhood of the union of segments $L_1$ and $L_2$. Informally speaking, we move the endpoints of  $\eta$ to the centers of the two $n$-gons, ``pushing'' all the parts of  $\eta$ intersecting  $L_1$ and $L_2$  along  $L_1$ and $L_2$ respectively, see Fig.~\ref{isotopy}. This isotopy takes $\eta$ to a non-self-intersecting curve $\eta'$ connecting the centers of the two $n$-gons and disjoint from $L_1$ and $L_2$ (and defining the same reflection in $G$). It is easy to see that if $\eta$ corresponds to a reflection $r\in G$, then the projection of $\eta'$ under the two-fold covering of $\O$ is precisely $\hat\gamma'_r$, which establishes a bijection between $\sigma$-admissible curves in~\cite{LL} and arcs in $D$.

\begin{figure}[!h]
\begin{center}
\psfrag{a}{\small (a)}
\psfrag{b}{\small (b)}
\epsfig{file=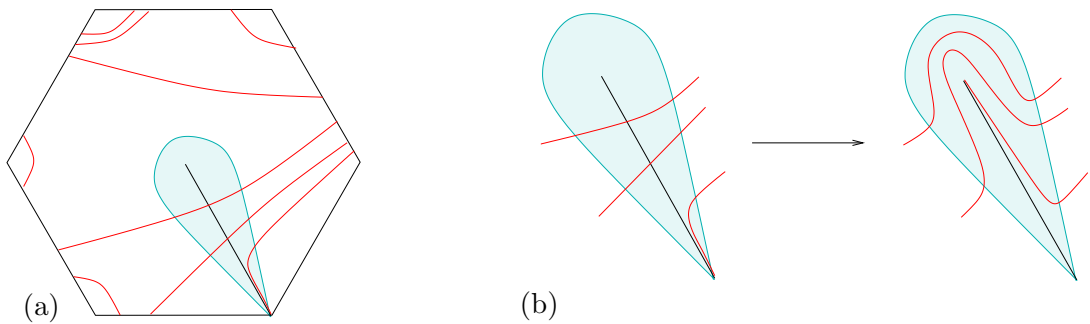,width=0.9\linewidth}
\caption{An isotopy of a $\sigma$-admissible curve~\cite{LL} projecting to $\hat\gamma_r$ in $D$.}
\label{isotopy}
\end{center}
\end{figure}

\begin{remark}
\label{is}

We would also like to note that there is another way to prove Theorem~\ref{thm-LL} without investigating the $Y$-seeds. Combining Theorem~\ref{disc} and Lemma~\ref{arc->root}, we see that we are left to show that every non-self-intersecting arc corresponds to a Schur root. In~\cite[Corollary 4.2]{IS}, Igusa and Schiffler proved that a real root is a Schur root if and only if the corresponding reflection is a prefix of the Coxeter element. Applying this result to our settings, we see that we only need to prove that every non-self-intersecting arc $\gamma$ can be included in an $n$-tuple of non-self-intersecting arcs, such that $\gamma$ is the first one in the clockwise order. This is of course always true since cutting $D$ along $\gamma$ gives rise to an $(n-1)$-punctured disc.    

This proof is shorter, but it does not say much about the structure of $Y$-seeds. We hope that the facts about bad pairs and  $Y$-seeds obtained while proving Theorem~\ref{seed<->1jump} and Corollary~\ref{arc->inside_seed} will be useful for answering further questions, in particular for understanding the compatibility of arcs (see Question~\ref{q}). 

\end{remark}

\subsection{Further questions}
\label{open}

Finally, we list several questions answers to which might be of interest. 

\begin{question}
Can the proof of Theorem~\ref{thm-LL} be extended to acyclic quivers that are not $2$-complete?
\end{question}

Our proof of Theorem~\ref{thm-LL} is based on the existence of an isomorphism between two groups, namely the Weyl group $G$ of the root system $\Delta$, and the group $\Gamma$ generated by order two rotations around some points on sides of an ideal hyperbolic polygon. If the initial quiver is not $2$-complete, the group $G$ is not a universal Coxeter group anymore.


\begin{question}
\label{q}
Let us call a collection of arcs {\em compatible} if they correspond to $\cc$-vectors belonging to one $Y$\!-\,seed. 
\begin{enumerate}
\item 
Which pairs of arcs are compatible?
\item 
Is a collection of mutually compatible arcs compatible itself?
\end{enumerate}
\end{question}

We note that the notion of compatibility will change if one considers $\dd$-vectors instead of $\cc$-vectors (as in~\cite{LL}): the collections of real Schur roots defining $n$-tuples of $\dd$-vectors of a seed and $\cc$-vectors of a $Y$\!-\,seed are essentially different.


\begin{thebibliography}{DWZ}
\bibitem[BGZ]{BGZ}  M.~Barot, C.~Geiss, A.~Zelevinsky, {\em Cluster algebras of finite type and positive symmetrizable matrices},  J. London Math. Soc. (2) 73 (2006), 545--564.
\bibitem[Bes]{Bes} D.~Bessis, {\em A dual braid monoid for the free group}, J. Algebra 302 (2006), 55--69.
\bibitem[Bir]{Bir} J.~Birman, {\em Braids, Links, and Mapping Class Groups}, 
Annals of Mathematics Studies No. 82, Princeton University Press (1975).

\bibitem[Bon]{B} K.~Bongartz, {\em Algebras and quadratics forms}, J. Lond. Math. Soc. 28 (1983), 461--469.

\bibitem[CK]{CK} P.~Caldero, B.~Keller, {\em From triangulated categories to cluster algebras. II}, Ann. Sci. \'Ecole Norm. Sup. (4) 39 (2006), 983--1009. 

\bibitem[CZ]{CZ} P.~Caldero, A.~Zelevinsky, {\em Laurent expansions in cluster algebras via quiver representations}, Mosc. Math. J. 6 (2006), 411--429.

\bibitem[C]{CB} W.~Crawley-Boevey, {\em Exceptional sequences of representations of quivers}, in ``Representations of algebras'', Proc. Ottawa 1992, eds V. Dlab and H. Lenzing, Canadian Math. Soc. Conf. Proc. 14 (Amer. Math. Soc., 1993), 117--124. 

\bibitem[DWZ]{DWZ} H.~Derksen, J.~Weyman, A.~Zelevinsky, {\em Quivers with potentials and their representations {\rm II}: Applications to cluster algebras}, J. Amer. Math. Soc. 23 (2010), 749--790.

\bibitem[FT]{rank3} A.~Felikson, P.~Tumarkin, {\em Geometry of mutation classes of rank $3$ quivers}, arXiv:1609.08828.

\bibitem[FZ]{FZ1} S.~Fomin, A.~Zelevinsky, {\em Cluster algebras {\rm I}: Foundations}, J. Amer. Math. Soc. 15 (2002), 497--529.

\bibitem[HK]{HK} A.~Hubery, H.~Krause, {\em A categorification of non-crossing partitions}, J. Eur. Math. Soc. 18 (2016), 2273--2313.

\bibitem[IS]{IS} K.~Igusa, R.~Schiffler, {\em Exceptional sequences and clusters}, J. Algebra 323 (2010), 2183--2202. 

\bibitem[Kac]{Kac} V.~Kac, {\em Infinite root systems, representations of graphs and invariant theory}, Invent. Math. 56 (1980), 57--92.

\bibitem[Kel]{K} B.~Keller, {\em On cluster theory and quantum dilogarithm identities}, Representations of algebras and related topics, 85--116, EMS Ser. Congr. Rep., Eur. Math. Soc., Z\"urich, 2011.

\bibitem[LL]{LL} K.-H.~Lee, K.~Lee, {\em A conjectural description for real Schur roots of acyclic quivers}, arXiv:1703.09113. 

\bibitem[M]{M} B.~Maskit, {\em On Poincar\'e's theorem for fundamental polygons}, Adv. Math. 7 (1971), 219--230.

\bibitem[N]{N} K.~Nagao, {\em Donaldson -- Thomas theory and cluster algebras}, Duke Math. J. 162 (2013), 1313--1367.

\bibitem[NCh]{NCh} A.~ N{\'a}jera Ch{\'a}vez, {\em On the $\cc$-vectors of an acyclic cluster algebra}, Int. Math. Res. Notices 2015 (2015), 1590--1600. 

\bibitem[Sch]{Sch} A.~Schofield, {\em General representations of quivers}, Proc. Lond. Math. Soc. 65 (1992), 46--64.

\bibitem[Se]{S2} A.~Seven, {\em Cluster algebras and symmetric matrices}, Proc. Amer. Math. Soc. 143 (2015), 469--478. 

\bibitem[ST]{ST} D.~Speyer, H.~Thomas, {\em Acyclic cluster algebras revisited}, Algebras, quivers and representations, Abel Symp., vol. 8, Springer, Heidelberg, 2013, pp. 275--298.

\bibitem[V]{V} E.~B.~Vinberg, {\em Discrete linear groups generated by reflections}, Math. USSR Izv. 5 (1971), 1083--1119. 


\bibitem[W]{W} M.~Warkentin, {\em Exchange graphs via quiver mutation}, Ph.D. thesis, 2014. Available at \\\url{http://www.qucosa.de/urnnbn/urn:nbn:de:bsz:ch1-qucosa-153172}

\end{thebibliography}
\end{document}